\documentclass[12pt,leqno]{article}
\usepackage{amsfonts}

\usepackage{amsmath, amsthm, amsfonts, amssymb, color, latexsym, enumerate}
\usepackage{mathrsfs}
\usepackage{url}

\usepackage{amsfonts}
\newtheorem{thm}{Theorem}[section]
\newtheorem{cor}[thm]{Corollary}
\newtheorem{lem}[thm]{Lemma}

\newtheorem{rem}[thm]{Remark}
\theoremstyle{definition}

\newcommand{\scr}[1]{\mathscr #1}
\definecolor{wco}{rgb}{0.5,0.2,0.3}

\numberwithin{equation}{section} \theoremstyle{remark}
\newcommand{\ua}{\uparrow}

\def\R{\mathbb R}
\def\Z{\mathbb Z}
\def\M{\mathcal{M}}
\def\N{\mathbb N}
\def\C{\mathbb C}
\def\g{\widetilde{\nabla}}
\def\E{\mathbb E}
\def\EE{\mathbb E}
\def\F{\mathscr{F}}

\def\d{\mathrm{d}}
\def\dd{d_{\text{s}}}
\def\B{B_{\text{s}}}
\def\e{\mathrm{e}}

\def\<{\langle} \def\>{\rangle}

\def\i{{\rm i}}

\def\W{\mathsf W}

\def\eps{\varepsilon}

\DeclareMathOperator{\Lip}{Lip}

\DeclareMathOperator{\supp}{supp}

\newcommand{\bra}[1]{\left( #1 \right)}

\newcommand{\sqa}[1]{\left[ #1 \right]}

\newcommand{\abs}[1]{\left| #1 \right|}

\newcommand{\les}{\preceq}

\title{{\bf Sharp  $L^q$-Convergence Rate in $p$-Wasserstein Distance for Empirical Measures of Diffusion Processes }\footnote{Supported in
 part by  the National Key R\&D Program of China (No. 2022YFA1006000, 2020YFA0712900) and NNSFC (11921001).} }
\author{{\bf    Feng-Yu Wang$^{a)}$, Bingyao Wu$^{b)}$, Jie-Xiang Zhu$^{a)}$   }\\
\footnotesize{ $^{a)}$Center for Applied Mathematics, Tianjin University, Tianjin 300072, China}\\
 \footnotesize{$^{b)}$School of Mathematics and Statistics,
 Fujian Normal University, Fuzhou 350007, China}\\
\footnotesize{    wangfy@tju.edu.cn, bingyaowu@163.com,  jiexiangzhu7@gmail.com}}

\begin{document}
\allowdisplaybreaks

\def\R{\mathbb R}  \def\ff{\frac} \def\ss{\sqrt} \def\B{\mathbf
B}
\def\N{\mathbb N} \def\kk{\kappa} \def\m{{\bf m}}
\def\ee{\varepsilon}\def\ddd{D^*}
\def\dd{\delta} \def\DD{\Delta} \def\vv{\varepsilon} \def\rr{\rho}
\def\<{\langle} \def\>{\rangle}
  \def\nn{\nabla} \def\pp{\partial} \def\E{\mathbb E}
\def\d{\text{\rm{d}}} \def\bb{\beta} \def\aa{\alpha} \def\D{\scr D}
  \def\si{\sigma} \def\ess{\text{\rm{ess}}}\def\s{{\bf s}}
\def\beg{\begin} \def\beq{\begin{equation}}  \def\F{\scr F}
\def\Ric{\mathcal Ric} \def\Hess{\text{\rm{Hess}}}
\def\e{\text{\rm{e}}} \def\ua{\underline a} \def\OO{\Omega}  \def\oo{\omega}
 \def\tt{\tilde}\def\[{\lfloor} \def\]{\rfloor}
\def\cut{\text{\rm{cut}}} \def\P{\mathbb P} \def\ifn{I_n(f^{\bigotimes n})}
\def\C{\scr C}      \def\aaa{\mathbf{r}}     \def\r{r}
\def\gap{\text{\rm{gap}}} \def\prr{\pi_{{\bf m},\varrho}}  \def\r{\mathbf r}
\def\Z{\mathbb Z} \def\vrr{\varrho} \def\ll{\lambda}
\def\L{\scr L}\def\Tt{\tt} \def\TT{\tt}\def\II{\mathbb I}
\def\i{{\rm in}}\def\Sect{{\rm Sect}}  \def\H{\mathbb H}
\def\M{\mathbb M}\def\Q{\mathbb Q} \def\texto{\text{o}} \def\LL{\Lambda}
\def\Rank{{\rm Rank}} \def\B{\scr B} \def\i{{\rm i}} \def\HR{\hat{\R}^d}
\def\to{\rightarrow}\def\l{\ell}\def\iint{\int}\def\gg{\gamma}
\def\EE{\scr E} \def\W{\mathbb W}
\def\A{\scr A} \def\Lip{{\rm Lip}}\def\S{\mathbb S}
\def\BB{\scr B}\def\Ent{{\rm Ent}} \def\i{{\rm i}}\def\itparallel{{\it\parallel}}
\def\g{{\mathbf g}}\def\Sect{{\mathcal Sec}}\def\T{\mathbb T}\def\BB{{\bf B}}
\def\f{\mathbf f} \def\g{\mathbf g}\def\BL{{\bf L}}  \def\BG{{\mathbb G}}
\def\Bd{{D^E}} \def\BdP{D^E_\phi} \def\Bdd{{\bf \dd}} \def\Bs{{\bf s}} \def\GA{\scr A}
\def\Bg{{\bf g}}  \def\Bdd{\psi_B} \def\supp{{\rm supp}}\def\div{{\rm div}}
\def\ddiv{{\rm div}}\def\osc{{\bf osc}}\def\1{{\bf 1}}\def\BD{\mathbb D}\def\GG{\Gamma}
\def\H{{\bf H}}
\maketitle
\def\R{\mathbb R}  \def\ff{\frac} \def\ss{\sqrt} \def\B{\mathbf
B}
\def\N{\mathbb N} \def\kk{\kappa} \def\m{{\bf m}}
\def\ee{\varepsilon}\def\ddd{D^*}
\def\dd{\delta} \def\DD{\Delta} \def\vv{\varepsilon} \def\rr{\rho}
\def\<{\langle} \def\>{\rangle}
  \def\nn{\nabla} \def\pp{\partial} \def\E{\mathbb E}
\def\d{\text{\rm{d}}} \def\bb{\beta} \def\aa{\alpha} \def\D{\scr D}
  \def\si{\sigma} \def\ess{\text{\rm{ess}}}\def\s{{\bf s}}
\def\beg{\begin} \def\beq{\begin{equation}}  \def\F{\scr F}
\def\Ric{\mathcal Ric} \def\Hess{\text{\rm{Hess}}}
\def\e{\text{\rm{e}}} \def\ua{\underline a} \def\OO{\Omega}  \def\oo{\omega}
 \def\tt{\tilde}\def\[{\lfloor} \def\]{\rfloor}
\def\cut{\text{\rm{cut}}} \def\P{\mathbb P} \def\ifn{I_n(f^{\bigotimes n})}
\def\C{\scr C}      \def\aaa{\mathbf{r}}     \def\r{r}
\def\gap{\text{\rm{gap}}} \def\prr{\pi_{{\bf m},\varrho}}  \def\r{\mathbf r}
\def\Z{\mathbb Z} \def\vrr{\varrho} \def\ll{\lambda}
\def\L{\scr L}\def\Tt{\tt} \def\TT{\tt}\def\II{\mathbb I}
\def\i{{\rm in}}\def\Sect{{\rm Sect}}  \def\H{\mathbb H}
\def\M{\mathbb M}\def\Q{\mathbb Q} \def\texto{\text{o}} \def\LL{\Lambda}
\def\Rank{{\rm Rank}} \def\B{\scr B} \def\i{{\rm i}} \def\HR{\hat{\R}^d}
\def\to{\rightarrow}\def\l{\ell}\def\iint{\int}\def\gg{\gamma}
\def\EE{\scr E} \def\W{\mathbb W}
\def\A{\scr A} \def\Lip{{\rm Lip}}\def\S{\mathbb S}
\def\BB{\scr B}\def\Ent{{\rm Ent}} \def\i{{\rm i}}\def\itparallel{{\it\parallel}}
\def\g{{\mathbf g}}\def\Sect{{\mathcal Sec}}\def\T{\mathbb T}\def\BB{{\bf B}}
\def\f{\mathbf f} \def\g{\mathbf g}\def\BL{{\bf L}}  \def\BG{{\mathbb G}}
\def\Bd{{D^E}} \def\BdP{D^E_\phi} \def\Bdd{{\bf \dd}} \def\Bs{{\bf s}} \def\GA{\scr A}
\def\Bg{{\bf g}}  \def\Bdd{\psi_B} \def\supp{{\rm supp}}\def\div{{\rm div}}
\def\ddiv{{\rm div}}\def\osc{{\bf osc}}\def\1{{\bf 1}}\def\BD{\mathbb D}\def\GG{\Gamma}
\def\H{{\bf H}}
\begin{abstract}  For a class of (non-symmetric) diffusion processes  on a length space, which in particular include the (reflecting) diffusion processes on a connected compact Riemannian manifold,
the exact   convergence rate is derived for $(\E[\W_p^q(\mu_T,\mu)])^{\ff 1 q} (T\to\infty)$ uniformly in $(p,q)\in [1,\infty)\times (0,\infty)$, where $\mu_T$ is the empirical measure of the   diffusion process, $\mu$ is the unique invariant probability measure, and $\W_p$ is the $p$-Wasserstein distance. Moreover,
when the  dimension parameter is less than $4$,  we prove that   $\E|T\W_2^2(\mu_T,\mu)-\Xi(T)|^q\to 0$  as $T\to\infty$ for any $q\ge 1$, where $\Xi(T)$ is  explicitly given by
  eigenvalues and eigenfunctions  for the symmetric part of the  generator.    \end{abstract} \noindent

 \vskip 2cm

\section{Introduction}

  In recent years, the convergence rate and renormalization limit in Wasserstein distance have been intensively studied
for the empirical measure of different models,  among other references we would like to mention \cite{W21,W22,WAAP,WJEMS} for diffusion processes on manifolds with reflected or killed boundary,  \cite{LW1,LW2,LW3,WW}
for subordinated diffusion processes, and  \cite{DU,HT} for McKean-Vlasov SDEs and fractional Brownian motion on torus.

In this paper, we study  the empirical measure  of non-symmetric diffusion processes on a length space introduced in the recent work \cite{Wang23NS},
and derive the exact  convergence rate in   $L^q(\P)$ for  $p$-Wasserstein distance uniformly in   $p,q\in [1,\infty)$.
To see the progress made in the present paper,  in the moment we simply consider the  diffusion processes on a compact Riemannian manifold.

Let $M$ be a $d$-dimensional compact connected Riemannian manifold possibly with a $C^2$-smooth boundary $\pp M$, let $\DD$ be the Laplace-Beltrami operator and $b$ be a $C^1$-vector field on $M$.
It is well known that the  diffusion process $X_t$ (with reflecting boundary   if $\pp M$ exists) generated by
$L:= \DD + b$  is exponentially ergodic, with a unique invariant probability measure $\mu(\d x)= \e^{V(x)}\d x$ for some $C^2$-function $V$ on $M$, where $\d x$ is the volume measure.  By the law of large numbers,   $\mu$ is
approximated   by the empirical measures $\mu_T$ as $T\to\infty$, where
$$\mu_T:= \ff 1 T\int_0^T \dd_{X_t}\d t,\ \ T>0, $$
and  $\dd_{X_t}$ is  the Dirac measure at the point $X_t$. We investigate the convergence rate of $\mu_T$ to $\mu$ under the $p$-Wasserstein distance for  $p\in [1,\infty)$:
\beq\label{WP} \W_p(\nu,\nu'):=\inf_{\pi\in \C(\nu,\nu')} \bigg(\int_{M\times M} \rr (x,y)^p\, \pi (\d x,\d y) \bigg)^{\ff 1 p},\ \ \nu,\nu'\in \scr P,\end{equation}
where $\scr P$ is the set of all probability measures on $M$,   $\C(\nu,\nu')$ is the class of all couplings of $\nu$ and $\nu'$, and $\rr$ is the Riemannian distance on $M$.

Write $L= \hat L+Z$, where $\hat L:=\DD+\nn V$ is the symmetric part of $L$, and   $Z:=b-\nn V$
is the anti-symmetric part with
  ${\rm div}_\mu\, Z=0,$  i.e.
   $$\int_M (Z f)\d\mu=0,\ \ f\in C^1(M).$$
According to \cite{Wu2}, for every $0\ne f\in L^2(\mu)$ with $\mu(f):=\int_Mf\d\mu=0,$ we have
$$ \sqrt{T} \mu_T(f)\to N(0, 2 {\bf V}(f))\ \text{ weakly\  as\  } T\to\infty,$$  where $N(0, 2 {\bf V}(f))$ is the centered Normal distribution with variance $2 {\bf V}(f)$,   ${\bf V}(f)$ is defined by
$${\bf V}(f):=\int_0^\infty \mu(fP_t f)\d t\in (0,\infty),$$
and $(P_t)_{t \geq 0}$ is the Markov semigroup for $X_t$, i.e.
$$P_tf(x):= \E^x[f(X_t)],\ \ t\ge 0,\ x\in M,\ f\in\B_b(M).$$
Here, and in what follows, $\E^x$ stands for the expectation taken for the underly Markov process starting from $x$ at time $0$. In general, for any probability measure $\nu$ on the state space,
 $\E^\nu$  denotes the expectation taken for the underly Markov process  with initial distribution $\nu$.

 For two positive variables $A$ and $B$, we denote $A\preceq B$ if $A\le c B$ holds for some constant $c>0$,
and write $A\asymp B$ if $A\preceq B$ and $B\preceq A$.
   \cite[Theorem 1.2]{Wang23NS} provides  the following estimates on the convergence rate of $\W_p(\mu_T,\mu):$
\beg{enumerate}\item[$\bullet$] When $d=1$, for any $(p,q)\in [1,\infty)\times [1,\infty)$, then for large $T>0$
 \beq\label{AA} (\E^\nu[\W_p^{2q}(\mu_T,\mu)])^{\frac 1 q}  \asymp T^{-1}\text{\  uniformly\  in\ } \nu\in \scr P.\end{equation}
 \item[$\bullet$] When $d=2$, \eqref{AA} holds  provided
$ p\in [1,\infty)$ and $q\in [1,\ff {p}{p-1})$.
\item[$\bullet$] When $d=3$, \eqref{AA} holds  provided
    $p\in [1,\ff 3 2)$ and $q\in [1,\ff{3p}{5p-3}).$
\item[$\bullet$] When $d=4$, $\E^\nu[\W_2^2(\mu_T,\mu)]\asymp T^{-1}\log T$ for large $T$ uniformly in $\nu\in \scr P.$
\item[$\bullet$] There exists a constant $c\in (1,\infty)$ such that
$$\sup_{\nu\in\scr P} \big(\E^\nu[\W_{2p}^{2q}(\mu_T,\mu)]\big)^{\ff 1 q}\le c T^{-\ff 2{d(3-p^{-1}-q^{-1})-2}},
\ \ T\ge 1$$
provided $d=4$ and $(p,q)\in [1,\infty)\times [1,\infty)\setminus\{(1,1)\}$, or $d\ge 5$ and $(p,q)\in [1,\infty)\times [1,\infty).$\end{enumerate}

Moreover, when $d\le 3$, and the boundary $\pp M$ is either empty or convex, $T\W_2^2(\mu_T,\mu)$ converges to $\Xi(T)$ in $L^q(\P)$ for some $q\ge 1$ as $T\to\infty$, where
\beq\label{PS} \Xi(T):= \sum_{i=1}^\infty \ff{|\psi_i(T)|^2}{\ll_i},\ \ \  \psi_i(T):=\ff 1 {\ss T} \int_0^T \phi_i(X_t)\d t,\end{equation}
for $\{\ll_i\}_{i\ge 0}$ and $\{\phi_i\}_{i\ge 0}$ being all eigenvalues and unitary eigenfunctions of $-\hat L:=-(\DD+\nn V)$ in $L^2(\mu)$ respectively, with Neumann boundary condition if $\pp M$ exists.

According to \cite[Theorem 1.1]{Wang23NS},  we have the following assertions:

\beg{enumerate}\item[$\bullet$] When $d\in\{1,2\}$, for any $q\in [1, \ff{2d}{(3d-4)^+})$,  where $\ff{2d}{(3d-4)^+}=\infty$ for $d=1$,
\beq\label{CT1} \lim_{T\to\infty} \sup_{\nu\in \scr P} \E^\nu\big[|T\W_2^2(\mu_T,\mu)-\Xi(T)|^q\big]=0.\end{equation}
\item[$\bullet$] When $d=3$,  then for any $k\in (\ff 3 2,\infty]$, $R\in [1,\infty)$ and $q\in [1,\ff 65)$,
\beq\label{CT2-} \lim_{T\to\infty} \sup_{\nu\in \scr P_{k,R}} \E^\nu\big[|T\W_2^2(\mu_T,\mu)-\Xi(T)|^q\big]=0,\end{equation}
where $\scr P_{k,R}:= \{\nu=h_\nu\mu\in \scr P:\ \|h_\nu\|_{L^k(\mu)}\le R\},$ and $\nu=h_\nu \mu$ stands for a probability measure having density $h_\nu$ with respect to $\mu$. \end{enumerate}

 These results are now considerably  improved by the following Theorem \ref{T1}, which is a consequence of the general results Theorem \ref{TN1},  Corollary \ref{CN}, Theorem \ref{TN2}, Theorem \ref{TN3}
 and Corollary \ref{C6.5} stated  in Section 2, since  assumptions used in these results
hold for the (reflecting) diffusion  generated by $L:=\DD+\nn V+Z$ on a compact compact Riemannian manifold,
where {\bf (A)}, \eqref{SCT} and {\bf (C)}  are well known, and {\bf (B)} is included in \cite[Lemma 5.2]{WZ}.

\begin{thm} \label{T1} Let $p\in [1,\infty)$ and $ q\in  (0,\infty)$. Then the following assertions hold.
\beg{enumerate}\item[$(1)$]  For any  $(k, R) \in (1, \infty] \times [1, \infty)$,
\begin{equation}\label{AS}
\big(\E^\nu[\W_p^q(\mu_T, \mu)]\big)^{\ff 2 q} \asymp \begin{cases}
T^{-1}, &\textrm{if } d \leq 3,\\
T^{-1}  \log (1+T), &\textrm{if } d = 4,\\
T^{-\frac{2}{d - 2}}, &\textrm{if } d \geq 5\end{cases}\end{equation} holds for large $T>0$ uniformly in $\nu\in \scr P_{k,R}.$
\item[$(2)$] When $p\le \ff{2d}{(d-2)^+} \lor \ff{d(d-2)}2,$  $\eqref{AS}$  holds uniformly in $\nu\in \scr P.$
\item[$(3)$] Let  $\pp M$ be either empty or convex.  When $d\in \{1,2\}$, $\eqref{CT1}$ holds for  any $q\in (0,\infty)$;  while when $d=3$, $\eqref{CT2-}$ holds for all  $(k, R) \in (1, \infty] \times [1, \infty)$ and $q\in (0,\infty)$.
\item[$(4)$]   Let $\pp M$ be either empty or convex,   let $d\le 3$ and $Z=0$. Then for any   $q \in (0,\infty) $   and $\nu \in \scr P$,
 $$
\lim_{T \to \infty} \E^\nu \big|T \W_2^2(\mu_T,\mu) \big|^q= \E\bigg|\sum_{i=1}^\infty \ff {2\xi_i^2}{\ll_i^2}\bigg|^q\in (0,\infty),
$$
where $\{\xi_i\}_{i\ge 1}$ are  i.i.d. standard random variables.  \end{enumerate}  \end{thm}

 Besides diffusion processes on compact Riemannian manifolds considered above, results for other models  presented  in \cite[Section 5]{Wang23NS} are also improved by using our general results stated in Section 2. To save space we do not illustrate   these models in details.

The remainder of the paper is organized as follows. In Section 2, we introduce the framework and main results of the paper. In Section 3, we   make some preparation estimates which will be used in Sections 4-6 to prove  the main results stated in Section 2.

\section{Main results under a general framework}

We first recall the framework studied in \cite{Wang23NS}, then state the main results derived in the present paper.

\subsection{The framework}
  Let  $(M,\rr)$ be a  bounded length space,   let $\scr P$ be the set of all probability measures on $M$, let $\B_b(M)$ be  the class of bounded measurable functions on $M$, and let $C_{b,{\rm L}}(M)$ be the set of all bounded Lipschitz continuous functions on $M$.
 For any
$p\in [1,\infty)$, the $p$-Wasserstein distance  $\W_p$ is defined as in \eqref{WP}.
Let $\hat X_t$ be a reversible Markov process on $M$ with the unique invariant probability measure $\mu\in \scr P$ having full support $M$.  Throughout the paper, we simply denote $\mu(f)=\int_M f\d\mu$ for $f\in L^1(\mu).$

The Markov semigroup $\hat P_t$ of $\hat X_t$ is formulated as
$$\hat P_t f(x)=\E^x[f(\hat X_t)],\ \ t\ge 0,\,  x\in M, \, f\in \B_b(M),$$
where     $\E^x$ stands for the expectation
for the underlying Markov process starting at point $x$. In general, for any $\nu\in \scr P$,  $\E^\nu$ stands for  the expectation for the underlying   Markov process with initial distribution $\nu$.

Let $(\hat \EE,\D(\hat \EE))$ and
$(\hat L, \D(\hat L))$ be, respectively,  the associated symmetric Dirichlet form and self-adjoint generator in $L^2(\mu)$.
We assume that $C_{b,{\rm L}}(M)$ is a dense subset of $\D(\hat\EE)$ under the   $\hat \EE_1$-norm
$\|f\|_{\hat\EE_1}:= \ss{\mu(f^2)+\hat\EE(f,f)},$
and
$$\hat\EE(f,g)=\int_M \GG(f,g)\d\mu,\ \ f,g\in C_{b,{\rm L}}(M)$$ holds for
a symmetric   square field (carr\'e du champ operator)
$$\GG: C_{b,{\rm L}}(M)\times C_{b,{\rm L}}(M)\to \B_b(M),$$
  such that for any $f,g,h\in C_{b,{\rm L}}(M)$ and $\phi\in C_b^1(\R),$  we have
\beg{align*} &\ss{\GG(f,f)(x)}=|\nn f(x)|:=\limsup_{y\to x} \ff{|f(y)-f(x)|}{\rr(x,y)},\ \ x\in M,\\
&\GG(fg,h)= f\GG(g,h)+ g\GG(f,h),\ \ \ \GG(\phi(f), h)=  \phi'(f) \GG(f,h).\end{align*}
Moreover, we assume that $\hat L$ satisfies the chain rule
$$\hat L\Phi(f)= \Phi'(f)\hat Lf +\Phi''(f) |\nn f|^2,\ \ \ f\in \D(\hat L)\cap C_{b,{\rm L}}(M), \, \Phi\in C^2(\R).$$

Let
$$Z: C_{b,{\rm L}}(M)\to \B_b(M)$$
be a bounded vector field  with ${\rm div}_\mu Z=0$, i.e. it satisfies
\beg{align*} &Z(fg)= fZg+gZ f,\ \   Z(\phi(f))= \phi'(f) Zf,\ \ \ f,g  \in C_{b,{\rm L}}(M),\ \phi\in C^1(\R),\\
  & \|Z\|_\infty:= \inf\big\{K\ge 0: \ |Zf|\le K |\nn f|,\ f\in C_{b,{\rm L}}(M)\big\}<\infty,\\
 & \mu(Zf):=\int_M (Z f) \d\mu=0,\ \ \ f\in C_{b,{\rm L}}(M). \end{align*}
 Consequently, $Z$ uniquely extends to a bounded linear operator from $\D(\hat \EE)$ to $L^2(\mu)$ satisfying
%\beq\label{RM} \mu(Zf) =0,\ \ \ f\in \D(\hat \EE),  \end{equation}
 \beq\label{RM} \mu(fZg) =-\mu(gZf),\ \ \ f,g\in \D(\hat\EE),\end{equation}
and
$$\EE(f,g):=\hat \EE(f,g)+\mu(f Zg),\ \ \ f,g\in \D(\EE)=\D(\hat\EE)$$
is a (non-symmetric) conservative   Dirichlet form with generator
$$L:=\hat L+Z,\ \ \ \D(L)=\D(\hat L),$$ which satisfies the chain rule
$$L\Phi(f)= \Phi'(f)Lf +\Phi''(f) |\nn f|^2,\ \ \ f\in \D(\hat L)\cap C_{b,{\rm L}}(M), \, \Phi\in C^2(\R).$$
Assume that  $L$ generates a unique diffusion process  $X_t$ on $M$. The associated Markov semigroup is given by
$$P_t f(x)= \E^x[f(X_t)],\ \ x\in M, \, t\ge 0, \, f\in \B_b(M).$$
By Duhamel's formula, $P_t$ and $\hat P_t$ have the following relation
\beq\label{00} P_tf= \hat P_tf+\int_0^t P_s\{Z\hat P_{t-s}f\}\d s,\ \ f\in \D(\hat \EE), \, t\ge 0.\end{equation}

We investigate the convergence of empirical measures
$\mu_T:=\ff 1 T\int_0^T \dd_{X_t}\d t$ to $\mu$ under the Wasserstein distance $\W_p$.

\subsection{Upper bound estimate}

To  estimate $(\E [\W_p^q(\mu_T,\mu)])^{\ff 1 q}$ from above,  we make the following assumption.

\beg{enumerate}\item[{\bf (A)}]   The following conditions hold for some $ d\in [1,\infty)$ and  an increasing function $K: [2,\infty)\to (0,\infty)$.
\item[$\bullet$] {\bf  Nash inequality.} There exists a constant $C>0$ such that
\beq\label{N}  \mu(f^2)\le C   \hat \EE (f,f)^{\ff d{d+2}}\mu(|f|)^{\ff 4{d+2}},\ \ f\in\D_0:=\big\{f\in\D( \hat \EE):\ \mu(f)=0\big\}. \end{equation}
\item[$\bullet$]  {\bf Continuity of symmetric diffusion.} For any $p\in [2,\infty)$,
\beq\label{CT}   \E^\mu [\rr(\hat X_0, \hat X_t)^p] =   \int_{  M\times M}  \rr(x,y)^p \hat p_t(x,y) \mu(\d x) \mu(\d y)\le K(p) t^{\ff p 2},\ \ t\in [0,1],\end{equation}
where $\hat p_t$ is the heat kernel of $\hat P_t$ with respect to $\mu$.
\item[$\bullet$] {\bf  Boundedness of Riesz transform.} For any $p\in [2,\infty)$,
\beq\label{RZ} \|\nn (-\hat L)^{-\ff 1 2} f\|_{L^p(\mu)}\le K(p)\|f\|_{L^p(\mu)},\ \ f\in L^p(\mu) \text{ with } \mu(f)=0.\end{equation}
 \end{enumerate}

It is well known that {\bf (A)} holds for the  (reflecting) diffusion process generated by $\hat L:=\DD+\nn V$ considered in Introduction.

Besides the elliptic diffusion process on compact manifolds,  some  criteria on the Nash inequality \eqref{N} are available in  \cite[Section 3.4]{Wbook}. In general,  \eqref{N} implies that for some constant $c_0 >0$,
 \beq\label{HS} \|\hat P_t-\mu\|_{L^p(\mu)\to L^q(\mu)}\le c_0 (1\wedge t)^{-\ff {d(q-p)} {2pq}}\e^{-\ll_1 t},\ \ t>0, \ 1 \leq p \leq q \leq \infty,\end{equation}
 and that $-\hat L$ has purely discrete spectrum with all eigenvalues $\{\ll_i\}_{i\ge 0}$, which are listed in the increasing order counting multiplicities, satisfy
  \beq\label{EG} \ll_i\ge c_1 i^{\ff 2 d},\ \ i\ge 0,\end{equation} for some constant $c_1>0$. The Markov semigroup $\hat P_t$ generated by $\hat L$ has symmetric heat kernel $\hat p_t$ with respect to $\mu$ formulated as
\beq\label{HS2} \hat p_t(x,y)= 1+ \sum_{i= 1}^\infty \e^{-\ll_i t} \phi_i(x)\phi_i(y),\ \ t>0, \, x,y\in M. \end{equation}
All these assertions can be found  for instance in \cite{Davies}.

The condition \eqref{CT} is natural for diffusion processes due to the growth property $\E|B_t-B_0|^p \le c t^{\ff p2}$ for  the Brownian motion $B_t$. There are plentiful results on the boundedness  condition \eqref{RZ} for the  Riesz transform, see \cite{B87, CD11, CTW23} and references therein.

The following result shows that  under assumption {\bf (A)}, the convergence rate of $(\E \W_p^q(\mu_T,\mu))^{\ff 1 q}$ is  given by
 $$\gg_d(T):= \beg{cases} T^{-\ff 1 2 }, &\textrm{if } \ d\in [1,4),\\
T^{-\ff 1 2}  \ss{\log T}, &\textrm{if } \ d = 4,\\
T^{-\frac{1}{d - 2}}, &\textrm{if }\  d\in (4,\infty).\end{cases} $$

\beg{thm}\label{TN1} Assume {\bf (A)}.
 Then   for any $ (k, p, q)\in (1,\infty] \times [1,\infty)\times (0,\infty)$, there exists a constant $c\in (0,\infty)$ such that
\begin{equation}\label{UP1}
\big(\E^\nu[\W_p^q(\mu_T, \mu)] \big)^{\ff 1 q} \le c \|h_\nu\|_{L^k(\mu)}^{\ff 1 q} \gg_d(T),\ \ T\ge 2, \ \nu= h_\nu\mu\in \scr P \text{ with } h_\nu \in L^k(\mu), \end{equation}
where $\nu= h_\nu\mu$ means $\ff{\d\nu}{\d\mu}=h_\nu.$

 \end{thm}

To show that $\gg_d(T)$ is   the convergence rate of  $\big(\E^\nu[\W_p^q(\mu_T, \mu)]\big)^{\ff 1 q}$ uniformly in $\nu\in \scr P$, we   restrict $p$ to  the set
$$I(d):=[1,\infty)\cap \bigg[1, \ff{2d}{(d-2)^+}\lor \ff{d(d-2)}2\bigg]=\begin{cases} [1,\infty),\ &\text{if} \ d\in [1,2],\\
[1, \ff{2d}{d-2}],  \ &\text{if} \ d\in (2,4],\\
[1,\ff{d(d-2)}2], \ &\text{if} \ d\in(4,\infty).\end{cases} $$

\beg{cor}\label{CN} Assume {\bf (A)} with $\eqref{CT}$ strengthened as
\beq\label{SCT} \sup_{x\in M} \E^x [\rr(\hat X_0, \hat X_t)^{p}]\le K(p) t^{\ff p 2},\ \ p\in [1,\infty),\ t\in [0,1].\end{equation}
Then for any $p\in I(d)$ and  $q \in (0, \infty)$,  there exists a constant $c \in (0,\infty)$ such that
\begin{equation}\label{UP1'}
\sup_{\nu\in \scr P} \E^\nu[\W_p^q(\mu_T, \mu)] \le c  \gg_d(T)^q,\ \ T\ge 2.\end{equation}
   When  $p \notin I(d)$, for any  $q \in(0, \infty)$   and $\bb \in (0, 1)$, there exists a constant $c \in (0, \infty)$ such that
 \begin{equation}\label{UP1''}
\sup_{\nu\in \scr P} \E^\nu[\W_p^q(\mu_T, \mu)] \le c   \gg_d(T)^{\bb q},\ \ T\ge 2.\end{equation}
 Consequently, for any $p\in [1,\infty)$ and $q \in (0, \infty)$,
$$\limsup_{T\to\infty} \ff 1 {\log (T^{-1})} \log\bigg[ \sup_{\nu\in \scr P} \big(\E^\nu[\W_p^q(\mu_T, \mu)]\big)^{\ff 1 q}\bigg]\le  \ff 1 {2\lor (d-2)}.$$  \end{cor}

\subsection{Lower bound estimate}

To derive sharp lower bound estimates, we make the following assumption {\bf (B)} which holds in particular for the (reflecting) diffusion operator  $\hat L:=\DD+\nn V $   on a $d$-dimensional compact connected Riemannian manifold, since in this case conditions \eqref{EG2} and \eqref{CT2} are well known,
and  the other conditions have been verified by   \cite[Lemma 5.2]{WZ}.  For $M$ being a smooth domain in $\R^d$,  \eqref{SA} is known as  Sard's lemma (see \cite[p130, Excercise 5.5]{Gry}) and
has been discussed in \cite[Section 3.1.6]{BLG}. The function $f_\xi$ in \eqref{LU} is called   Lusin's approximation  of $h$ (see \cite{AF, Liu}).

\beg{enumerate}\item[{\bf (B)}]    Let $\{\ll_i\}_{i\ge 0}$ be all eigenvalues of  $-\hat L$ listed in the increasing order with multiplicities. There exist  constants $\kk>0$ and $d\in [1,\infty)$ such that
\beq\label{EG2}  \ll_i \le \kk\,  i^{\ff  2 d},\ \ \ i\ge 0,\end{equation}
\beq\label{CT2} \W_1(\nu \hat P_t,\mu)\le \kk\, \W_1(\nu,\mu),\ \ t\in [0,1], \, \nu\in \scr P.\end{equation}
Moreover,   for any $f\in\D(\hat\EE)$,
\beq\label{SA} \mu\big(\{|\nn f|>0, f=0\}\big)=0,\end{equation}
and there exists a constant $c>0$ independent of $f$ such that
\beq\label{LU}
\mu \big(   f  \ne f_\xi   \big) \leq \frac{c}{\xi^2} \int_M |\nabla f|^2 \d \mu,\ \ \xi>0
\end{equation} holds for   a family of functions $\{f_\xi:\xi>0\}$ on $M$ with   $\|\nn f_\xi \|_{\infty} \leq \xi$.
\end{enumerate}

 \beg{thm}\label{TN2} Assume {\bf (A)} and {\bf (B)}. Then for any  $q\in (0,\infty)$, we have for large $T>0$,
 $$\inf_{\nu\in \scr P} \E^\nu[\W_1^q(\mu_T,\mu)]   \succeq \gg_d(T)^q.$$
 \end{thm}

\subsection{ Convergence of $T\W_2^2(\mu_T,\mu)$ in $L^q(\P)$}

As explained in \cite{Wang23NS} that    the following assumption  holds for diffusion processes
with Bakry-Emery curvature bounded from below.

\beg{enumerate} \item[{\bf (C)}] $(M,\rr)$ is a geodesic space, there exist constants  $\theta \in (0, 1]$, $K>0$ and $m\ge 1$ such that
 \begin{equation*} |\nn \hat P_t\e^f|^2\le (\hat P_t \e^f)\hat P_t(|\nn f|^2\e^f)+ K t^\theta \|\nn f\|_\infty^2 (\hat P_t\e^{2mf})^{\ff 1 m},
 \ \ t\in [0,1], \,  f\in C_{b,{\rm L}}(M),\end{equation*}
 and there exists   a function $h\in C([0,1];[1,\infty))$ with $h(0)=1$ such that
\beq\label{B22}\W_2^2(\nu \hat P_r,\mu)\le h(r) \W_2^2(\nu,\mu),\ \ \nu\in\scr P, \, r\in [0,1].\end{equation}
Moreover, \eqref{N} holds.
\end{enumerate}

Let  $\psi_i(T)$ and $\Xi(T)$ be in \eqref{PS}.  We have the following result.

\beg{thm}\label{TN3} Assume that   for any $p\in (1,\infty)$ there exists a constant $k(p)\in (0,\infty)$ such that
\beq\label{GE-}|\nn \hat P_tf|\le k(p) (\hat P_t|\nn f|^p)^{\ff 1 p},\ \ t\in [0,1], \, f\in C_{b,{\rm L}}(M).\end{equation}
  \beg{enumerate}
 \item[$(1)$]  Assume {\bf (A)}. If   $d \in [1, 4)$, then  for any $(k, R) \in (1, \infty] \times [1, \infty)$,
\begin{equation}\label{UP3}
     \lim_{T\to\infty}\sup_{\nu\in \scr P_{k,R} } \E^\nu\Big[ |\{T\W_2^2(\mu_T,\mu) - \Xi(T)\}^+|^q \Big] = 0,\ \ q\in [1,\infty).
 \end{equation}
Moreover, if $d \in [1, \frac83)$, then for any   $q\in [1, \ff{d}{2(d-2)^+})$,  where $\ff{d}{2(d-2)^+}: = \infty$ for $d \le 2$,
  \begin{equation}\label{TUP3}
     \lim_{T\to\infty}\sup_{\nu\in \scr P  } \E^\nu\Big[ |\{T\W_2^2(\mu_T,\mu) - \Xi(T)\}^+|^q \Big] = 0.
 \end{equation}

 \item[$(2)$] Assume   {\bf (A)} and {\bf (C)}. If $d \in [1, 4)$, then for any $(k, R) \in (1, \infty] \times [1, \infty)$,
 \beq\label{LM}   \lim_{T\to\infty}\sup_{\nu\in \scr P_{k,R}} \E^\nu\Big[|\{T\W_2^2(\mu_T,\mu)-\Xi(T)\}^-|^q\Big]=0 \end{equation}
holds for
   $$
 q  \in \beg{cases} [1,\infty), &\text{if}\ d\in [1,3],\\
\big [1, \ff{d-2}{2k^*(d-3)}\big), &\text{if}\ d\in (3,4),\end{cases}
 $$
 where $k^*$ is the H\"older conjugate of $k$.  Moreover, if $d \in [1, \frac83)$, then for any $q\in [1, \ff{d}{2(d-2)^+})$,
 \beq\label{TLM}   \lim_{T\to\infty}\sup_{\nu\in \scr P } \E^\nu\Big[|\{T\W_2^2(\mu_T,\mu)-\Xi(T)\}^-|^q\Big] = 0.\end{equation}
\end{enumerate}
 \end{thm}

 To derive the uniform convergence uniformly in $\nu\in\scr P$ also for $d\in (2,4)$, we modify $\Xi(T)$ by  the following defined $\bar \Xi_r(T)$ and $\tt\Xi_r(T)$ for a constant $r\in (0,\infty)$:
 \beq\label{NNM} \beg{split}
&\bar \psi_{i,r}(T) : = \frac{1}{\sqrt{T-r}} \int_{r}^{T} \phi_i(X_{t}) \d t, \quad \bar \Xi_r(T) : = \sum_{i = 1}^{\infty} \frac{|\bar\psi_{i,r}(T)|^2}{\ll_i},\ \ T>r,\\
&\tt \psi_{i,r}(T) : = \frac{1}{\sqrt{T}} \int_{0}^{T} \phi_i(X_{r+t}) \d t, \quad \tt \Xi_r(T) : = \sum_{i = 1}^{\infty} \frac{|\tt \psi_{i,r}(T)|^2}{\ll_i},\ \ T>0.\end{split}\end{equation}

 \begin{thm} \label{TN4} In the situation of Theorem $\ref{TN3}$, let $r\in (0,\infty)$.
 \beg{enumerate}
 \item[$(1)$] Assume {\bf (A)}.  If   $d \in [1, 4)$, then for any $q\in [1,\infty)$,
\begin{equation}\label{UP33}
     \lim_{T\to\infty}\sup_{\nu\in \scr P  } \E^\nu\Big[|\{T\W_2^2(\mu_T,\mu) -\bar \Xi_{r} (T)\}^+|^q+|\{T\W_2^2(\mu_T,\mu)-\tt \Xi_{r} (T)\}^+|^q \Big]=0.
 \end{equation}
   \item[$(2)$] Assume   {\bf (A)} and {\bf (C)}. If $d \in [1, 4)$, then for any
 $$
 q \in \beg{cases} [1,\infty), &\text{if}\ d\in [1,3],\\
[1,\ff{d-2}{2(d-3)}), &\text{if}\ d\in (3,4),\end{cases}
 $$
   \beq\label{TLM3}   \lim_{T\to\infty}\sup_{\nu\in \scr P } \E^\nu\Big[|\{T\W_2^2(\mu_T,\mu) - \bar\Xi_r(T)\}^-|^q+|\{T\W_2^2(\mu_T,\mu)-\tt\Xi_r(T)\}^-|^q\Big]=0. \end{equation}
\end{enumerate}
 \end{thm}

As a consequence of Theorem \ref{TN4}, we have the following result on the weak limit of $T \W_2^2(\mu_T,\mu),$ which leads to the limits of moments.

\beg{cor}\label{C6.5} Assume {\bf (A)} and {\bf (C)} for  $d \in [1, 4)$, let $Z=0$ and  $\eqref{GE-}$ holds.
 Then for any initial distribution $\nu$ of $X_t$, as $T\to\infty,$ $T\W_2^2(\mu_T,\mu)$
 converges weakly to the real random variable
 $$\Xi(\infty):= \sum_{ i=1}^\infty\ff{2\xi_i^2}{\ll_i^2},$$
 where $\{\xi_i\}_{i\ge 1}$ are i.i.d.  standard Normal random variables on $\R$.
 Consequently, for any
 $$
 q \in \beg{cases} (0,\infty), &\text{if}\ d\in [1,3],\\
(0, \ff{d-2}{2(d-3)}), &\text{if}\ d\in (3,4),\end{cases}
 $$
 \beq\label{QAP} \lim_{T\to\infty} \E^\nu [T^q\W_2^{2q}(\mu_T,\mu)] =\E\big|\Xi(\infty)\big|^q<\infty. \end{equation}
 \end{cor}

\beg{proof} By \eqref{EG}, we have $\sum_{i=1}^\infty\ff 1 {\ll_i^2}<\infty$ for $d < 4.$ So,  as shown in the proof of \cite[Lemma 2.11]{WZ}, we may prove   the weak convergence
of $T\W_2^2(\mu_T,\mu)$ to $\Xi(\infty)$ using Theorem \ref{TN4}. By combining this with \eqref{82} below,  we are able to apply the dominated convergence theorem to get \eqref{QAP}.

\end{proof}

\section{Preparations}

We first introduce the following  deviation inequality due to  \cite[Theorem 1]{wu2000deviation}.

\begin{lem} \label{LDT}  Let $X_t$ be an ergodic Markov process on a Polish space $E$ with generator $L$  and unique invariant probability measure $\mu$,
and let
 $${\scr E}_{{\rm sym}}(f,g):= \frac 1 2 \big[ \<- Lf,g\>_{L^2(\mu)}+\<-Lg,f\>_{L^2(\mu)} \big],\ \ f,g \in\D(L)$$
be the associated  symmetrized  Dirichlet form with domain $\D({\scr E}_{{\rm sym}}) \subset L^2(\mu)$.
For any $\xi\in (0,\infty)$ and $g \in L^1(\mu)$ with  $\mu(g) = 0$, let
$$
I_g(\xi-) := \lim_{\varepsilon \downarrow 0}\inf \left\{ {\scr E}_{{\rm sym}}(h,h):  \, h \in  \D({\scr E}_{{\rm sym}}), \, \mu(h^2) = 1, \, |\mu(g h^2)| = \xi-\vv, \, \mu(|g| h^2 ) < \infty \right\},
$$ where  $\inf \emptyset = \infty$ by convention.
Then
$$
\mathbb{P}^\nu \bra{ \left| \frac1T \int_{0}^{T} g(X_t) \d t \right| > \xi } \leq 2  \left\| h_\nu \right\|_{L^2(\mu)} \exp \Big[- T I_g(\xi-) \Big],\ \ T,\xi>0, \, \nu=h_\nu\mu\in \scr P.
$$
\end{lem}

We now apply this lemma to establish a Bernstein-type inequality for the local framework introduced in Subsection 2.1.
For any $g \in L^2(\mu) $ with $\mu(g)=0$, let
\beq\label{SM} \beg{split}
&\sigma^2( g ) := 2 \| (- \hat L)^{-\frac12} g \|_{L^2(\mu)}^2,\\
& \mathfrak{m}(g) := \begin{cases} \|(-\hat L)^{-\ff 1 2} g \|_{L^2(\mu)}, & \textrm{if } d \in [1,2),\\
 \inf_{2 < p < \infty} \frac{p}{p - 2} \| \nabla (-\hat L)^{-1} g \|_{L^p(\mu)}, &\textrm{if } d  = 2,\\
 \| \nabla (-\hat L)^{-1} g \|_{L^d(\mu)}, &\textrm{if } d\in(2,\infty),\end{cases}\\
\end{split}\end{equation}
In the spirit of \cite[Theorem 2.2]{gao2014bernstein}, we have the following consequence of Lemma \ref{LDT}.

\begin{cor}[Bernstein-type inequality] \label{C3.2} Consider the framework in Subsection 2.1 such that  $\eqref{N}$
holds. Then there exists  a constant  $\aa  > 0$ such that for any
 $g \in \D((-\hat L)^{-\frac12})$ with $\mu(g)=0$, any $\nu \in \mathscr{P}$ such that $h_\nu:=\ff{\d\nu}{\d\mu}$ exists,
\begin{equation} \label{bernstein}\begin{split}
&\mathbb{P}^\nu \bra{  \left| \frac 1T \int_{0}^{T} g(X_t) \d t \right| > \xi }\\
  & \leq 2   \left\| h_\nu \right\|_{L^2(\mu)}  \exp \bigg[- \frac{  T \xi^2}{ 2 \sigma^2(g) \, + \, \aa  \mathfrak{m}(g) \xi } \bigg],
   %\\ &\leq 4^{\ff{k-1}k}  \left\| h_\nu \right\|_{L^k(\mu)} \exp \bigg[- \frac{2(k-1)T \xi^2}{k(2 \sigma^2(g) \, + \, \aa_2 \mathfrak{m}_2(g) \xi)} \bigg],
   \ \ \ T, \xi>0.\end{split}
\end{equation}
Consequently, for any $p\in [1,\infty)$ there exists a constant  $C(p) >0$ such that
\begin{equation} \label{in:Rosenthal}
 \E^\nu \sqa{ \left| \frac1T \int_{0}^{T} g(X_t) \d t \right|^p } \leq C(p) \left\| h_\nu \right\|_{L^2(\mu)} \left(  \sigma(g)  T^{-\ff 1 2} +  \mathfrak{m}(g) T^{-1}  \right)^p.
\end{equation}
\end{cor}

\begin{proof} (a)  We first observe that \eqref{bernstein} implies \eqref{in:Rosenthal}.
Indeed, by taking
$$r= \ff{ T\xi^2}{ 2\si^2(g) +\aa \mathfrak{m} (g)\xi},$$
we have
$$\xi = \ff{r  \aa  \mathfrak{m}(g) +\ss{[r  \aa  \mathfrak{m} (g)]^2+ 8  Tr  \si^2(g) }}{2 T}.$$
So,  we find   constants $c_1, c_2>0$ depending on $p$ such that \eqref{bernstein} yields
 \begin{align*} %\label{distribution}
&\E^\nu \sqa{ \left| \frac1T \int_{0}^{T} g(X_t) \d t \right|^p } = p \int_{0}^{\infty} \xi^{p-1} \mathbb{P}^\nu \bra{  \left| \frac1T \int_{0}^{T} g(X_t) \d t \right| > \xi } \d \xi\\
&\le c_1 \left\| h_\nu \right\|_{L^2(\mu)} \int_0^\infty  \Big( \ff{\mathfrak{m} (g) r}{ T} + \ff{\si(g)\ss r}{\ss{ T}}\Big)^{p-1} \Big(\ff{\mathfrak{m} (g)}{ T}+ \ff{\si(g)}{\ss{ rT}}\Big)\e^{-r}\d r\\
&\le c_2  \left\| h_\nu \right\|_{L^2(\mu)} \left(  \ff{\sigma(g) }{\ss{ T} }+ \ff{\mathfrak{m} (g) } { T  }  \right)^p.\end{align*}
 Therefore,  it remains to prove  \eqref{bernstein}.

(b)  According to \eqref{RM}, the symmetrized Dirichlet form ${\scr E}_{{\rm sym}}$ in Lemma \ref{LDT} is exactly $\hat{\scr E}$ in our setting.  By Lemma \ref{LDT},    \eqref{bernstein} follows  if we find a constant $\aa >0$ such  that  for any  $h\in \D(\hat{\EE})$  with $\mu(h^2)=1$,
$$
\frac{|\mu(gh^2)|^2}{2\sigma^2(g)   +   \aa \mathfrak{m}(g) |\mu(gh^2)| } \leq \hat{\EE}(h,h),$$
and this inequality is implied by
\beq\label{0*} |\mu(gh^2)|\le \ss{2\si^2(g) \hat{\EE}(h,h)} + \ff \aa 2  \mathfrak{m} (g) \hat{\EE}(h,h).
\end{equation}
 Since $C_{b,{\rm L}}(M)$ is dense in $\D(\hat{\EE})$, we may assume that $h\in C_{b,{\rm L}}(M)$. To verify this estimate, let $\hat h=h-\mu(h)$. Then $\mu(g)=0$ implies
\begin{align} \label{target}
\mu(g h^2) = 2 \mu(h) \mu(g \hat h) + \mu \big( g \hat h^2).
\end{align}
By $|\mu(h)| \leq \mu(h^2)^{\frac12} = 1,$ and  noting that the Cauchy-Schwarz inequality implies
\begin{align*}
|\mu(g \hat h)| = | \mu \big( \{ (-\hat L)^{-\frac12} g \}\cdot \{(-\hat L)^{\frac12}\hat h \}\big)| \leq \sqrt{\frac{\sigma^2(g) \hat{\EE}(h,h)}{2}},
\end{align*}
  we obtain
\begin{align} \label{part1}
|2 \mu(h)\mu(g \hat h)| \leq \sqrt{2 \sigma^2(g) \hat{\EE}(h,h)}.
\end{align}
By \eqref{target} and \eqref{part1},  \eqref{0*} follows  from
\beq\label{*1*}
\big| \mu\big( g \hat h^2  \big) \big| \leq \ff \aa 2 \mathfrak{m}(g) \hat{\EE}(h,h).
\end{equation}
Moreover, we have
 \beq\label{IV}  (-\hat L)^{-\aa}= \ff 1 {\GG(\aa)}  \int_0^\infty t^{\aa-1} \hat P_t \, \d t,\ \ \aa>0,\end{equation}
 where $\GG$ is the Gamma function.
With \eqref{HS} and \eqref{IV} in hands,    we prove   estimate  \eqref{*1*} below for $d\in [1,2),d=2$ and $d>2$ respectively.

(c) Let $d\in [1,2)$. By \eqref{N} and \eqref{HS} we find   constants $c_1,c_2>0$ such that
\beg{align*} &\|\hat h\|_{L^\infty(\mu)}=\|(-\hat L)^{-\ff 1 2}(-\hat L)^{\ff 1 2} \hat h\|_{L^\infty(\mu)}\\
& = \ff 1 {\GG(\ff 1 2)} \bigg\|\int_0^\infty t^{-\ff 1 2 }\hat P_t (-\hat L)^{\ff 1 2 }\hat h \, \d t\bigg\|_{L^\infty(\mu)}\\
&\le c_1   \int_0^\infty t^{-\ff 1 2} (1 \wedge t)^{-\frac{d}4}\e^{-\ll_1 t} \d t \cdot \|(-\hat L)^{\ff 1 2 }\hat h\|_{L^2(\mu)}= c_2 \ss{\hat{\EE}(h,h)}. \end{align*}
Combining this with the fundamental inequalities
$$\hat \EE(f,g)^2\le \hat{\EE}(f,f)\hat{\EE}(g,g),\ \  \hat \EE(f^2,f^2)\le 4\|f\|_{L^\infty(\mu)}^2 \hat{\EE}(f,f),$$
we derive
\beg{align*}& |\mu(g \hat h^2)|= \big|\hat{\EE}\big((-\hat L)^{-1} g, \hat h^2 \big)\big|\\
&\le \ss{\hat{\EE}\big((-\hat L)^{-1} g, (-\hat L)^{-1} g\big) \hat{\EE} \big(\hat h^2, \hat h^2\big)}
 \\
&\le 2 \|(-\hat L)^{-\ff 1 2} g\|_{L^2(\mu)} \|\hat h\|_{L^\infty(\mu)}  \ss{\hat{\EE}(h,h)}
 \le 2 c_2  \mathfrak{m}(g) \hat{\EE}(h,h),\end{align*}
 so that \eqref{*1*} holds for $\aa = 4 c_2.$

 (d) Let $d=2$.
By the Cauchy-Schwarz inequality, the chain rule and H\"older's inequality, we have for any $p \in (2, \infty)$,
\begin{equation} \label{SQ}  \begin{split}
&|\mu(g \hat h^2)|= \big|\hat{\EE} \big( (-\hat L)^{-1} g, \hat h^2 \big)\big|\\
& \leq  \mu \left( \Gamma\big( (-\hat L)^{-1} g,  (-\hat L)^{-1} g  \big)^{\frac12} \cdot \Gamma(\hat h^2, \hat h^2)^{\frac12}    \right) \\
&\le  2  \mu \left( \Gamma\big( (-\hat L)^{-1} g,  (-\hat L)^{-1} g  \big)^{\frac12} \cdot |\hat{h}|  \Gamma(h, h)^{\frac12}    \right)  \\
& \le 2\|\nn(-\hat{L})^{-1}g\|_{L^p(\mu)}  \|\hat{h}\|_{L^{\frac{2p}{p - 2}}(\mu)} \sqrt{\hat \EE(h, h)}.
 \end{split} \end{equation}
 Moreover, by \eqref{HS} for $d=2$, we find   constants $c_3,c_4>0$ such that for any $p \in (2, \infty)$,
\beg{align*} &\|\hat h\|_{L^{\ff{2p}{p-2}}(\mu)}= \|(-\hat L)^{-\ff 1 2}(-\hat L)^{\ff 1 2}\hat h\|_{L^{\ff{2p}{p-2}}(\mu)}= \ff 1 {\GG(\ff 1 2)} \bigg\|\int_0^\infty t^{-\ff 1 2}\hat P_t (-\hat L)^{\ff 1 2}\hat h \, \d t\bigg\|_{L^{\ff{2p}{p-2}}(\mu)}\\
&\le c_3 \int_0^\infty t^{-\ff 1 2} (1 \wedge t)^{-\frac1p} \e^{-\ll_1 t} \d t \cdot \|(-\hat L)^{\ff 1 2}\hat h\|_{L^2(\mu)}
 \le \ff{c_4 p}{p-2}\ss{\hat{\EE}(h,h)}.\end{align*}
Combining this with \eqref{SQ} and the definition of $\mathfrak{m}(g)$ for $d=2$, we derive \eqref{*1*}  for $\aa = 4 c_4$.

(e) Finally, let   $d >2$.   It is well known that when $d>2$,  the Nash inequality \eqref{N} implies the Sobolev inequality (cf. \cite{Davies})
$$  \| f \|_{L^{\frac{2d}{d-2}}(\mu)} \leq C \sqrt{ \hat \EE (f, f)},\ \ f\in\D(\hat\EE) \text{ with } \mu(f)=0.  $$
Combining this inequality with \eqref{SQ} for $p = d$ yields
$$|\mu(g \hat h^2)|  \le
2\|\nn(-\hat{L})^{-1}g\|_{L^d(\mu)}  \|\hat{h}\|_{L^{\frac{2d}{d - 2}}(\mu)} \sqrt{\hat \EE(h, h)}
\le 2C \|\nn(-\hat{L})^{-1}g\|_{L^d(\mu)} \hat{\EE}(h,h).
$$ Then  \eqref{*1*} holds  for  $\aa = 4C$, which completes the proof.
\end{proof}

\begin{rem}  Note that \eqref{N} implies  that for any $1 < p < d$  satisfying $\frac1q \geq \frac1p  - \frac1d,$ there exists a constant $c>0$ such that  (see \cite{V85})
  \begin{align*}
\| (- \hat L)^{-\frac12} f \|_{L^q(\mu)} \leq C_{p, q} \| f \|_{L^p(\mu)}, \ \  f \in L^p(\mu) \text{ with } \mu(f) = 0.
\end{align*}
This together with  \eqref{RZ}  yields $\mathfrak{m}(g) < \infty$ for
\begin{equation*}
g \in
\begin{cases}
L^1(\mu) , & \textrm{if } d \in [1, 2);\\
L^{1 + 0}(\mu) : = \bigcup_{p > 1} L^p(\mu), & \textrm{if } d =2;\\
L^{\frac{d}{2}}(\mu),& \textrm{if } d >2.
\end{cases}
\end{equation*}
 \end{rem}

Recall that for every non-zero function $\phi \in L^2(\mu)$ with $\mu(\phi)=0$,
\beq \label{def of V}
{\bf V}(\phi):= \int_0^\infty \mu(\phi P_t \phi) \, \d t\in (0,\infty).
\end{equation}
We  have  the following lemma  on  the second moment of  $\psi_i(T)$ defined in \eqref{PS}.

\begin{lem} \label{lem:ito-tanaka-eigen} Consider the framework in Subsection 2.1 such that $\eqref{N}$ holds.
Then there  exist constants $c, c' > 0$  such that
 \beq \label{NM}
\Big|\E^\mu \big[ |\psi_i(T)|^2 \big] - \ff 2 {\ll_i}+ \ff 2 {\ll_i^2} {\bf V}(Z\phi_i)\Big| \le \ff c {\ll_i (1+T)},\ \ i\ge 1, \, T>0,
\end{equation}
\beq \label{VZ}
{\bf V}(Z\phi_i)  \le c' \ss{\ll_i}, \ \ i \geq 1.
\end{equation}
\end{lem}

\begin{proof}
(a) By Lemma 4.2(1) in \cite{Wang23NS}, we have
\begin{equation*}
{\bf V}(\phi_i) = \ff 1 {\ll_i}- \ff{1}{\ll_i^2}{\bf V}(Z\phi_i).
\end{equation*}
This together with the  Markov property and the stationarity of $X_t$  yields
\beq\label{*1}
\beg{split}  \E^\mu \big[ |\psi_i(T)|^2 \big] &= \ff 2 T\int_0^{T} \d t_1 \int_{t_1}^T \mu(\phi_i P_{t_2 - t_1} \phi_i) \d t_2 \\
&=  \ff 2 T\int_0^{T} \d t_1 \int_{0}^{T - t_1} \mu(\phi_i P_t \phi_i) \d t\\
 &= 2\Big(\ff 1 {\ll_i}-\ff 1{\ll_i^2} {\bf V}(Z \phi_i)\Big) - \ff 2 T \int_0^T \d t_1 \int_{T - t_1}^\infty \mu(\phi_i P_t \phi_i)\d t.
\end{split}
\end{equation}

We next evaluate the integrand $\mu(\phi_i P_t \phi_i)$. By Duhamel's formula \eqref{00} and $\hat P_t \phi_i = \e^{-\ll_i t} \phi_i$, we have
\begin{equation*}
P_t \phi_i = \e^{-\ll_i t} \phi_i + \int_0^t \e^{-\ll_i(t-s)} P_s (Z\phi_i)\d s,
\end{equation*}
so that
\begin{equation*}
\mu(\phi_i P_t\phi_i) =\e^{-\ll_i t} +\int_0^t \e^{-\ll_i (t-s)} \mu\big((P_s^*\phi_i) Z\phi_i\big) \d s,
\end{equation*}
where $P_t^*$ is the semigroup with generator $L^* : = \hat L - Z$.
Noting that $\mu(\phi_i Z \phi_i) = 0$ and using Duhamel's formula \eqref{00} again
$$P_t^* \phi_i = \e^{-\ll_i t} \phi_i - \int_0^t \e^{-\ll_i(t-s)} P_s^* (Z\phi_i)\d s,$$
we obtain
$$\mu\big((P_s^*\phi_i) Z\phi_i\big) = - \int_0^s \e^{-\ll_i(s-r)}\mu\big((Z\phi_i)P_r(Z\phi_i)\big)\d r.$$
Combining above identities with the fact that
$$ \big| \mu\big((Z\phi_i)P_r(Z\phi_i)\big) \big| \le \e^{-\ll_1 r}  \|Z\phi_i\|_{L^2(\mu)}^2\le \|Z\|_\infty^2 \ll_i \e^{-\ll_1 r},$$
we find a constant $C>0$ such that
\beg{align*}
\big| \mu(\phi_i P_t\phi_i) - \e^{-\ll_i t} \big| \le   \|Z\|_\infty^2 \ll_i\int_0^t \e^{-\ll_i(t-s)}\d s \int_0^s \e^{-\ll_i(s-r)-\ll_1r}\d r \le C \ll_i^{-1} \e^{-\frac{\ll_1}{2} t}.
\end{align*}
Combining with \eqref{*1} we finish the proof of \eqref{NM}.

(b) By  the spectral representation and that $\ll\e^{-\ll t}\le t^{-1}$ for $t,\ll>0$, we obtain
$$\|\nn \hat P_t f\|_{L^2(\mu)}^2 = \sum_{i=1}^\infty  \ll_i \e^{-2\ll_i t} \mu(f\phi_i)^2 \le t^{-1} \e^{-\ll_1 t} \|f\|_{L^2(\mu)}^2,\ \ f\in L^2(\mu).$$
Combining with  the Duhamel's formula  \eqref{00}, we derive
 $$ \|\nn P_t f\|_{L^2(\mu)} \le t^{-\ff 1 2 } \e^{-\frac{\ll_1 }{2} t} \|f\|_{L^2(\mu)} +\|Z\|_\infty \int_0^t s^{-\ff 1 2} \e^{- \frac{\ll_1}{2} s} \|\nn P_{t-s} f\|_{L^2(\mu)}\d s,\ \ t\ge 0.$$
By the generalized Gronwall inequality, see \cite{YGD}, we find constants $c_0,\ll>0$ such that
 $$ \| \nn P_t  f \|_{L^2(\mu)} \le   c_0  t^{-\frac12} \e^{-\ll t} \| f \|_{L^2(\mu)}, \ \ t > 0, \, f\in L^2(\mu). $$
Noting that $\|\nn\phi_i\|_{L^2(\mu)}^2=\ll_i$, by this and \eqref{RM}, we find a constant $c_1>0$ such that
$$ \big| \mu((Z\phi_i)P_t (Z\phi_i)) \big| = \big| \mu(\phi_i ZP_t (Z\phi_i)) \big| \le \|Z\|_\infty \| \nn P_t (Z\phi_i) \|_{L^2(\mu)} \le c_1  t^{-\frac12} \e^{-\ll t} \ss{\ll_i}.$$
Therefore, for some constant $c'>0$,
\begin{align*}
{\bf V}(Z\phi_i)= \int_0^\infty  \mu((Z\phi_i)P_t (Z\phi_i))\d t  \le c' \ss{\ll_i}.
\end{align*}
Then the proof is finished.
\end{proof}

Finally, we present  the following Lemma \ref{LC}  on the asymptotic   of $  \W_p(\cdot,\mu)$ near by the measure $\mu$, by using  the gradient estimate
\begin{align} \label{ii}
\| \nabla \hat P_t f \|_{L^p(\mu)} \leq c(p) t^{-\frac12} \| f \|_{L^p(\mu)}, \ \ t \in (0, 1), \, f \in C_{b, \rm{L}}(M).
\end{align} for $p\in (1,\infty)$ and some constant $c(p)>0$.
This result  extends
  \cite[Proposition 2.9]{GT24} and \cite[Proposition 2.5]{MT24+} for $p > \ff d {d-1}$   and $M=\R^d$.  The estimate for $p \le \ff d {d-1}$  is new even in the Euclidean setting.

We would like to clarify the links of   conditions \eqref{ii},  \eqref{RZ} and \eqref{GE-}. Firstly, \eqref{ii} is a necessary condition for the $L^p$-boundedness of the Riesz transform. Indeed, by the analyticity of $(\hat P_t)_{t > 0}$ on $L^p(\mu)$, the $L^p$-boundedness of the Riesz transform implies that
\begin{equation*}
\| \nabla \hat P_t f \|_{L^p(\mu)} \leq C_p \| (-\hat L)^{\frac12} \hat P_t f  \|_{L^p(\mu)} \leq C_p' \| f \|_{L^p(\mu)}, \ \ t > 0, \, f \in C_{b, \rm{L}}(M),
\end{equation*}
see \cite{ACDH04} for more   relationships between \eqref{ii} and \eqref{RZ}.   Secondly, according to \cite[Lemma 3.1(2)]{Wang23NS},  \eqref{GE-}  implies  \eqref{ii}   for all $p \in (1, \infty)$.

\begin{lem} \label{LC}
Consider the framework in Subsection 2.1 and let $p \in [1, \infty)$. If $p \in (1, \infty)$, assume furthermore \eqref{N}, \eqref{SCT} and \eqref{ii}. Then there exists a constant $c>0$  depending on $p$ and $d$ such that for any $\theta \in (0, 1)$,
\begin{align*}
\sup_{\nu \in \mathscr{P} }\W_p \big( (1 - \theta) \mu + \theta \nu, \mu \big) \le c  \cdot \begin{cases}
\theta, &\textrm{if } 1 \leq p < \frac{d}{d - 1},\\
\theta \log(1 + \theta^{-1}), &\textrm{if }  p = \frac{d}{d - 1},\\
\theta^{\ff 1 p+\ff 1 d}, &\textrm{if } p > \frac{d}{d - 1}.
\end{cases}
\end{align*}
\end{lem}

\begin{proof}
 We first recall that for any probability density function  $f \in L^2(\mu)$,
 \beq\label{AM0} \W_2^2(f\mu,\mu)\le \int_M\ff{|\nn \hat L^{-1} (f-1)|^2}{\scr M(f)}\,\d\mu,\ \ \scr M(f):= 1_{\{f>0\}}\ff{f-1}{\log f},\end{equation}
 \beq\label{L17} \W_p(f  \mu, \mu)  \le p\big\|\nn \hat L^{-1} (f -1)\big\|_{L^p(\mu)}.\end{equation}
 These estimates have been  presented in \cite{AMB} and \cite{L17} respectively by using the Kantorovich dual formula  and   Hamilton-Jacobi equations, which are available when $(M,\rr)$ is a   length space as we assumed, see \cite{V}.

By   the convexity of $\W_p^p$,  for any $\theta \in [0, 1]$ and $\mu_0, \mu_1, \nu_0, \nu_1 \in \mathscr{P},$  we have
\begin{align} \label{convexity}
\W_p^p((1 - \theta) \mu_0 + \theta \mu_1, (1 - \theta) \nu_0 + \theta \nu_1) \leq (1 - \theta) \W_p^p(\mu_0, \nu_0) + \theta \W_p^p(\mu_1, \nu_1).
\end{align}
This implies  the desired estimate for $p=1$. In the following we only consider      $p\in (1,\infty)$.

By an approximation argument, we may assume that $\nu$ has a density function $h_\nu$ with respect to $\mu$, so that
$$\nu_\vv:= \nu \hat P_\vv= (\hat P_\vv h_\nu)\mu,\ \ \vv\in (0,1).$$
Combining this with \eqref{SCT}, we find a constant $c_1>0$ independent of $\vv$ and $\nu$ such that
$$\W_p(\nu_\vv,\nu)\le c_1\ss\vv.$$
Combining this with      the triangle inequality, \eqref{L17} and \eqref{convexity}, we obtain
\begin{equation} \label{000} \beg{split}
&\W_p \big( (1 - \theta) \mu + \theta \nu, \mu \big) \leq \W_p \big( (1 - \theta) \mu + \theta \nu,  (1 - \theta) \mu +  \theta   \nu_\vv   \big) + \W_p \big( (1 - \theta) \mu +  \theta  \nu_\vv, \mu    \big)\\
&\le c_1 \theta^{\ff 1 p}\ss\vv + p \theta  \big\|\nn (-\hat L)^{-1} \hat{P}_{\vv}(h_\nu-1)\big\|_{L^p(\mu)},\ \  \vv,\theta\in (0,1).\end{split}
\end{equation}
To evaluate the term $\big\|\nn (-\hat L)^{-1} \hat{P}_{\vv}(h_\nu-1)\big\|_{L^p(\mu)}$, we need the gradient bound \eqref{ii}.
By combining  \eqref{ii} with \eqref{HS} for $(p,1)$ in place of $(p,q)$, and using the semigroup property of $\hat P_t$, we find a constant $C_p>0$ such that
\begin{equation} \label{GB}
\| \nabla \hat P_t f \|_{L^p(\mu)} \leq C_p (1 \wedge t)^{-\frac{d(p - 1)}{2p} - \frac12} \e^{-\ll_1 t} \| f \|_{L^1(\mu)}, \ \ t > 0, \, f \in L^1(\mu).
\end{equation}
This together with   \eqref{IV} for $\aa = 1$ yields that for  some constants $c_2, c_3 >0$,
\beq \label{DD'}\beg{split}&
\big\|\nn (-\hat L)^{-1} \hat{P}_{\vv} (h_\nu-1)\big\|_{L^p(\mu)} \le \int_{0}^{\infty} \| \nabla \hat  P_{t + \varepsilon} h_\nu \|_{L^p(\mu)} \d t \\
&\le c_2 \int_{0}^{\infty} [1 \wedge ( t + \eps)]^{-\ff{d(p-1)}{2p} - \frac12}  \e^{-\ll_1 t}\d t\\
& \le c_3 \begin{cases} 1, &\text{if}\ p<\ff d {d-1},\\
\log(1+\vv^{-1}), &\text{if}\ p=\ff d {d-1},\\
\vv^{\ff 1 2-\ff{d(p-1)}{2p}}, &\text{if}\ p>\ff d {d-1}.\end{cases}\end{split}
\end{equation}
Combining this with \eqref{000}, for  $\vv = \theta^2$ when $p<\ff d{d-1}$,  and for $\vv = \theta^{\ff 2{d}}$ when $p\ge \ff d {d-1}$, we derive the desired estimates.
\end{proof}

\section{Proofs of Theorem \ref{TN1} and Corollary \ref{CN}} \label{upper}

\beg{proof} [Proof of Theorem \ref{TN1}]    Let $\nu=h_\nu\mu$ with $h_\nu \in L^k(\mu)$ and $k \in (1, \infty]$. By the Markov property and H\"older's inequality, we obtain
\beg{align*}& \E^\nu \big[\W_p^q(\mu_T,\mu)\big]= \E^\mu\big[h_\nu(X_0) \W_p^q(\mu_T,\mu)\big]\\
&\le \big(\E^\mu[h_\nu^k]\big)^{\ff 1 k} \big(\E^\mu [\W_p^{k^* q}(\mu_T,\mu)]\big)^{\ff{1}{k^*}}= \|h_\nu\|_{L^k(\mu)} \big(\E^\mu [\W_p^{k^* q}(\mu_T,\mu)]\big)^{\ff{1}{k^*}},\end{align*}
where $k^*:=\ff k{k-1}$ is the H\"older conjugate of $k$. Let $p' = p \vee (k^* q) \vee 2$.  By H\"{o}lder's inequality and $\W_p\le \W_{p'}$, we derive
\beg{align*} \big(\E^\mu [\W_p^{k^* q}(\mu_T,\mu)]\big)^{\ff{1}{k^*}} \le \big(\E^\mu [\W_{p'}^{p'}(\mu_T,\mu)]\big)^{\ff{q}{p'}}.\end{align*}
 Therefore, it suffices to prove that for any  $p\in [2,\infty),$
\beq\label{UPB}   \E^\mu \big[\W_p^p(\mu_T,\mu)\big]\les \gg_d (T)^p,\ \ T\ge 2.\end{equation}

By \eqref{L17} and \eqref{RZ}, we obtain
\beq\label{L17'} \W_p^p(f\mu, \mu)  \le p^p \mu\big(|\nn \hat L^{-1} (f-1)|^p\big)\le (pK(p))^p\mu\big(|(-\hat L)^{-\ff 1 2} (f-1)|^p\big).\end{equation}
 To prove the upper bound estimate \eqref{UPB} using \eqref{L17'},  we approximate $\mu_T$ by the modified empirical measures
$$\mu_{T, \eps} :=\mu_T \hat{P}_\eps,\ \ \ T>0,\, \eps\in (0,1].$$
The  density function of $\mu_{T, \eps}$ with respect to $\mu$ is given by
\beq\label{FT}
f_{T, \eps}(y) : = \frac1T \int_{0}^{T} \hat p_\eps(X_t, y) \d t,\ \  y\in M.\end{equation}
 By the convexity of $\W_p^p$, the stationarity of $X_t$ and \eqref{CT}, we obtain
\beq \label{CH} \beg{split}
\E^\mu[\W_p^p(\mu_T, \mu_{T, \eps})]&\le \ff 1 T \int_0^T\E^\mu[\W_p^p(\dd_{X_t},\dd_{X_t}\hat{P}_{\eps})]\d t\\
&\le \int_{  M\times M}  \rr(x,y)^p \hat p_{\eps}(x,y) \mu(\d x) \mu(\d y) \les \eps^{\ff p 2}. \end{split}
\end{equation}
 Combining this with    \eqref{L17'}, we derive
\beq\label{GPP} \beg{split}&
\E^\mu[\W_p^p(\mu_T, \mu)]   \les  \E^\mu[\W_p^p(\mu_T, \mu_{T, \eps})] + \E^\mu[\W_p^p(\mu_{T, \eps}, \mu)]\\
&\les \vv^{\ff p2} + \E^\mu\big[\|(-\hat L)^{-\ff 1 2} (f_{T,\vv}-1)\|_{L^p(\mu)}^p\big],\ \ p \in [2, \infty).  \end{split}
 \end{equation}
Noting that
$$  (-\hat L)^{-\ff 1 2} (f_{T, \eps} - 1) (y)=  \frac1T \int_{0}^{T} (-\hat L_y)^{- \frac12}(\hat p_\eps(X_t, y) - 1) \d t,$$
where $\hat L_y$ is the operator $\hat L$ acting on the variable $y$, we deduce form  Corollary \ref{C3.2} that
\begin{equation}   \label{bound0}
 \E^\mu \left[ \|    (-\hat L)^{-\ff 1 2} (f_{T, \eps} - 1)  \|_{L^p(\mu)}^p \right]
  \les      T^{-\frac p2} A_\eps +  T^{-p} B_\eps,\ \
 T>0, \end{equation}
where
\begin{align*}
A_\eps : = \int_{M}  \sigma \big( (-\hat L_y)^{- \frac12}(\hat p_\eps(\cdot, y) - 1)  \big)^p     \mu(\d y),
\end{align*}
\begin{align*}
B_\eps : = \int_{M}  \mathfrak{m} \big( (-\hat L_y)^{- \frac12}(\hat p_\eps(\cdot, y) - 1) \big)^p   \mu(\d y).
\end{align*}

\  \newline
{\bf Estimate on $A_\eps$}. By the definition of $\sigma$,  \eqref{HS2}, $ -\hat L \phi_i= \ll_i\phi_i$, and the fact that $\{\phi_i\}_{i \geq 1}$ is orthonormal in $L^2(\mu)$, we have
\beq\label{AE}\beg{split}
&\sigma \big( (-\hat L_y)^{-\frac12}(\hat p_\eps(\cdot, y) - 1)  \big) =\ss 2 \big\| (- \hat L_y)^{-1} (\hat p_\eps(\cdot, y) - 1)  \big\|_{L^2(\mu)}\\
&= \bigg( 2 \int_{M} \bigg| (- \hat L_y)^{-1}  \sum_{i = 1}^{\infty} \e^{-\lambda_i \eps} \phi_i(x) \phi_i(y)   \bigg|^2  \mu(\d x)  \bigg)^{\frac12}\\
& = \bigg( 2 \sum_{i = 1}^{\infty} \frac{\e^{-2 \lambda_i \eps}}{\lambda_i^2} \phi_i^2(y) \bigg)^{\frac12}.
\end{split}\end{equation}
Combining this with  the  identity
\begin{align*}
\ll^{-\alpha} = \frac{1}{\Gamma(\alpha)} \int_{0}^{\infty} \e^{-\ll s} s^{\alpha - 1} \d s,\ \ \ \ll,\aa>0,
\end{align*}
 and applying \eqref{HS2}, we obtain that for any   $\eps \in (0, 1)$ and $y \in M$,
$$
\sum_{i = 1}^{\infty} \frac{\e^{-2 \lambda_i \eps}}{\lambda_i^2} \phi_i^2(y)  =  \ff 1 {\GG(2)} \int_{0}^{\infty} s (\hat p_{s + 2\eps}(y, y) - 1) \d s.$$
By \eqref{HS} for $p=1$ and $q=\infty$ we obtain
$$|\hat p_{s + 2\eps}(y, y) - 1|\les [1 \wedge (s+2\vv)]^{-\ff d 2} \e^{-\ll_1 s},\ \ \eps, s>0.$$
Therefore,  for $\vv\in (0,1]$,
\begin{align*}&\sum_{i = 1}^{\infty} \frac{\e^{-2 \lambda_i \eps}}{\lambda_i^2} \phi_i^2(y) \les \int_0^\infty s [1 \wedge (s+2\vv)]^{-\ff d 2} \e^{-\ll_1 s}\d s \\
&\les 1 + \int_{0}^{1} s(s + 2\eps)^{-\frac{d}{2}} \d s \nonumber \les  \begin{cases}  1, & \textrm{if } d \in [1,4),\\
  \log(1+\eps^{-1}), &\textrm{if } d = 4,\\
  \eps^{-\frac{d - 4}{2}}, &\textrm{if } d\in(4,\infty).\end{cases}
\end{align*}
Consequently,  for $\vv\in (0,1]$,
\begin{align} \label{bound1}
A_\eps \les  \begin{cases} 1, & \textrm{if } d \in [1,4),\\
 [\log(1+\eps^{-1})]^{\frac p2}, &\textrm{if } d = 4,\\
  \eps^{-\frac{(d - 4)p}{4}}, &\textrm{if } d\in(4,\infty).\end{cases}
\end{align}

\ \newline {\bf Estimate  on $B_\eps$}.
Let $d\in [1,2)$.  For any $y\in M$,  by the same reason leading to \eqref{AE} and \eqref{bound1} for $d<4$ , we obtain that for $\vv\in (0,1)$
\beg{align*}& B_\eps= \int_M \mathfrak{m} \big( (-\hat L_y)^{- \frac12}(\hat p_\eps(\cdot, y) - 1) \big)^p\mu(\d y) \\
&= \int_M \big\|(-\hat L_y)^{-1} (\hat p_\eps(\cdot, y) - 1)\big\|_{L^2(\mu)}^p\mu(\d y)\les 1.\end{align*}

Let $d=2$.  By taking $p=4$ in the definition of $\mathfrak{m}(\cdot)$, applying  \eqref{RZ} for $p = 4$,  \eqref{HS} for $p=1$ and $q=4$, and using \eqref{IV} for $\aa=1$, we obtain that for  $\vv\in (0,1)$,
\beg{align*} &B_\eps= \int_M \mathfrak{m} \big( (-\hat L_y)^{- \frac12}(\hat p_\eps(\cdot, y) - 1) \big)^p\mu(\d y)  \\
& \les  \sup_{y\in M} \big\|(-\hat L_y)^{-1} (\hat p_\eps(\cdot,y)-1)\big\|_{L^4(\mu)}^p \\
 &\le \sup_{y\in M} \bigg(\int_0^\infty \|\hat p_{t+\eps}(\cdot,y)-1\|_{L^4(\mu)}\d t\bigg)^p\\
 &\les \bigg(\int_0^\infty (1 \wedge t)^{-\ff {3d }{8}}\e^{-\ll_1 t}\d t \bigg)^p
 = \bigg(\int_0^\infty  (1 \wedge t)^{-\ff {3 }{4}}\e^{-\ll_1 t}\d t\bigg)^p < \infty.\end{align*}

 Let $d>2$. By the same argument used above, we have

\beg{align*} &B_\eps \les \sup_{y\in M}  \big\|(-\hat L_y)^{-1} (\hat p_\eps(\cdot,y)-1)\big\|_{L^d(\mu)}^p
  \le \sup_{y\in M} \bigg(   \int_0^\infty \|\hat p_{t+\eps}(\cdot,y)-1\|_{L^d(\mu)}\d t\bigg)^p\\
   & \les \bigg(\int_0^\infty [1 \wedge ( t + \eps)]^{-\ff {d-1 }{2}}\e^{-\ll_1 t} \d t\bigg)^p \les \begin{cases} 1,\ &\text{if}\ d\in (2, 3),\\
\log^p(1+\eps^{-1}),\ &\text{if}\ d=3,\\
  \eps^{-\ff{(d-3)p}2},\ &\text{if}\ d\in(3,\infty).\end{cases}   \end{align*}

So, in conclusion,
$$ B_\eps     \les \begin{cases} 1,\ &\text{if}\ d\in[1,3),\\
\log^p(1+\eps^{-1}),\ &\text{if}\ d=3,\\
  \eps^{-\ff{(d-3)p}2},\ &\text{if}\ d\in(3,\infty).\end{cases}  $$
Combining this with \eqref{bound0} and \eqref{bound1}, we derive  that for $T\ge 2$ and $\vv\in (0,1)$,
\beq\label{ED} \E^\mu\big[\|(-\hat L)^{-\ff 1 2} (f_{T,\vv}-1)\|_{L^p(\mu)}^p\big]\les \beg{cases} T^{-\ff p 2} + T^{-p},\ &\text{if}\ d\in[1,3),\\
T^{-\ff p 2} + T^{-p}  \log^p(1+\vv^{-1}),\ &\text{if}\ d=3, \\
T^{-\ff p 2} + T^{-p}\vv^{-\ff{(d-3)p}2},\ &\text{if}\ d\in (3,4),\\
T^{-\ff p 2} \log^{\ff p 2}(1+\vv^{-1})+ T^{-p}\vv^{-\ff{p}2},\ &\text{if}\ d=4,\\
T^{-\ff p 2}  \vv^{-\ff{(d-4)p}4}+ T^{-p}\vv^{-\ff{(d-3)p}2},\ &\text{if}\ d\in(4,\infty).\end{cases}\end{equation}
By taking  in this estimate
\begin{align*}
\eps = \begin{cases}  T^{-1}, & \textrm{if } d \in [1, 4),\\
 T^{-\frac{2}{d - 2}}, &\textrm{if } d \in [4,\infty),\end{cases}
\end{align*}
and combining with \eqref{GPP}, we derive \eqref{UPB}.  \end{proof}

 \beg{proof}[Proof of Corollary  \ref{CN}]  (a) By the monotonicity of $\W_p$ in $p$, it suffices to prove \eqref{UP1'} for $p \in I(d) \cap [2, \infty).$ Let
  $$\bar \mu_T:=\ff 1 {T-1}\int_1^T \dd_{X_t}\d t,\ \ T>1.$$
For any $\nu\in \scr P$,  let $\nu_1=\nu  P_1^*.$ Then \eqref{HS} implies
\beq\label{A0} \sup_{\nu\in \scr P} \ff{\d\nu_1}{\d\mu}\le c_0\end{equation}
 for some constant $c_0>0$, see \cite[Lemma 3.1]{Wang23NS}.  Hence,  by the Markov property and \eqref{UP1} for  $k = \infty$, we find a constant $c_1>0$ such that for any $T\ge 2$,
 \beq\label{A1} \sup_{\nu\in \scr P} \E^\nu[\W_p^q(\bar \mu_T,\mu)]=  \sup_{\nu\in \scr P} \E^{\nu_1}[\W_p^q(\mu_{T-1},\mu)]\le c_1 \gg_d(T)^q.\end{equation}
Noting that noting that $\mu_T= (1-T^{-1}) \bar \mu_T + T^{-1} \mu_1$, by the triangle inequality, \eqref{convexity} and Lemma \ref{LC},    we obtain
\beg{align*}  \W_p(\mu_T,\mu)& \le  \W_p\big((1-T^{-1})  \mu + T^{-1}\mu_1, \mu_T\big)+ \W_p\big((1-T^{-1})  \mu + T^{-1}\mu_1, \mu\big)\\
  &\le \W_p(\bar \mu_T,\mu)+ c_2 \gg_{d,p}(T)\end{align*}
 for some constant $c_2>0$,  where
 $$ \gg_{d,p}(T):= \beg{cases} T^{-1},\ &\text{if} \ p\in [1, \ff d{d-1}),\\
 T^{-1}  \log T,\ &\text{if}\ p=\ff d{d-1},\\
 T^{-(\ff 1 p+\ff 1 d)}, &\text{if}\ p\in(\ff d{d-1},\infty).\end{cases} $$
 Combining this with \eqref{A1},  we obtain
\beq\label{22}   \sup_{\nu\in \scr P}  \E^\nu [\W_p^q(\mu_T,\mu)] \les \big( \gg_d(T) +   \gg_{d,p}(T) \big)^q, \ \ T \geq 2.\end{equation}
Noting that $ \gg_{d,p}(T)\les \gg_d(T)$ holds for  $p\in I(d)$ and $T\ge 2$,  we derive \eqref{UP1'} by \eqref{22}.

(b) Now, let $p\notin I(d)$.   For any $s\in (0,1)$,    let $\nu_s= \nu P_s^*$. By \cite[Lemma 3.1]{Wang23NS}, \eqref{HS} implies that for a constant $c_1 > 0$
\begin{align} \label{HS3}
\| P_t \|_{L^p(\mu) \to L^q(\mu)} \leq c_1 t^{-\frac{d}{2}(p^{-1} - q^{-1})}, \ \ t \in (0, 1), \, 1 \leq p \leq q \leq \infty.
\end{align}
Then
\begin{equation} \label{A0'}
\sup_{\nu\in \scr P}\Big\| \ff{\d\nu_s}{\d\mu}\Big\|_{L^k(\mu)} \le \|  P_s \|_{L^{\frac{k}{k - 1}}(\mu) \to L^\infty(\mu)} \le c_1 s^{-\frac{d}{2}(1 - k^{-1})},\ \ s\in (0,1), \, k\in [1, \infty].
\end{equation}
For $s \in (0, 1)$, let
\begin{equation} \label{def of mu}
\tt \mu_{T,s} := \ff 1 T\int_s^{T+s}\dd_{X_t}\d t.
\end{equation}
By \eqref{UP1}, there exists a constant $c_2 > 0$ (depending on $k$, $p$ and $q$) such that
 \beq\label{N1} \sup_{\nu\in \scr P} \E^\nu [\W_p^q(\tt\mu_{T,s},\mu)] \le c_2  s^{-\ff{d(k-1)}{2k}}\gg_d(T)^q,\ \ s\in (0,1), \, T\ge 2.\end{equation}
On the other hand,    \eqref{N} implies that $D:=\sup \rr<\infty$, and  it is easy to see that for $T>s$,
$$\Pi :=\frac{1}{T}\int_{0}^{s}\big(\delta_{X_t}\times \delta_{X_{t+T}} \big) \d t+\frac{1}{T}\int_{s}^T\big(\delta_{X_t}\times \delta_{X_{t}} \big)\d t\in\mathscr{C}(\mu_T, \tt\mu_{T,s}).$$ Therefore for any $p\in [1,\infty)$,
\beq\label{DPS}
\W_p^p(\mu_T,\tt\mu_{T,s}) \le\int_{M\times M}\rho(x,y)^p \, \Pi(\d x,\d y)\le\frac{s D^p}{T},\ \ T>s.
\end{equation}
Combining this with \eqref{N1} and using the triangle inequality, we find a constant $c_3>0$ such that
 \begin{align*}
\sup_{\nu\in \scr P} \E^\nu[ \W_p^q(\mu_T,\mu) ] &\le 2^{q} \big( c_2 s^{-\ff{d(k-1)}{2k}}\gg_d(T)^q + D^q T^{-\ff q p} s^{\ff q p} \big)\\
& \le c_3 \big( s^{-\ff{d(k-1)}{2kq}}\gg_d(T) + \gg_d(T)^{\ff{2 \vee (d - 2)}{p}} s^{\ff 1 p}  \big)^q,\ \ s\in (0,1), \, T \geq 2,
\end{align*}
 %$$\sup_{\nu\in \scr P} \big(\E^\nu[\W_p^q(\mu_T,\mu) \big)^{\ff 1 q}\le  c_2  s^{-r}\gg_d(T)+ c_3 T^{-\ff 1 p} s^{\ff 1 p},\ \ s\in (0,1).$$
where $k \in (1, \infty)$ to be chosen later. Since for $p \notin I(d)$ we have $\ff{2 \vee (d - 2)}{p} < 1$, taking  $s= c_4 \gg_d(T)^{\ff{2kq [p - 2 \vee (d - 2)]}{2kq + d (k - 1)p}}$
 %$s= c_4\gg_d(T)^{\ff{p}{1+rp}}$
 for a sufficiently small $c_4>0$ yields
 $$\sup_{\nu\in \scr P} \E^\nu[\W_p^q(\mu_T,\mu)] \les  \gg_d(T)^{\ff {2kq + d(k - 1)[2 \vee (d - 2)]} {2kq + d(k - 1)p} \cdot q},\ \ T\ge 2.$$
 %$$\sup_{\nu\in \scr P} \big(\E^\nu[\W_p^q(\mu_T,\mu) \big)^{\ff 1 q}\le  c  \gg_d(T)^{\ff 1 {1+rp}},\ \ T\ge 2.$$
 Then the proof is finished by choosing   $k$ sufficiently close to $1$, such that $\bb\le \ff {2kq + d(k - 1)[2 \vee (d - 2)]} {2kq + d(k - 1)p}$.
 %small  $r\in (0,1)$ such that $\bb\le \ff 1 {1+rp}.$

 \end{proof}

\section{Proof  of Theorem \ref{TN2}  }

We first deal with the stationary case with initial distribution $\mu$.   Namely, we intend to prove that for any $q \in (0, \infty)$,
\begin{equation} \label{LB1}
\E^\mu [\W_1^q(\mu_T, \mu)] \succeq \gg_d(T)^q
\end{equation}
holds for large $T > 0$. The proof is organized into steps.

 (a)  By H\"older's inequality, if $q \in [1, \infty)$,
 \begin{equation*}
 \E^\mu [\W_1^q(\mu_T, \mu)] \geq \E^\mu [\W_1(\mu_T, \mu)]^q;
 \end{equation*}
 and if $q \in (0, 1)$,
 \begin{equation*} \beg{split}
 & \E^\mu[\W_1^q(\mu_T, \mu)] \geq \E^\mu [\W_1^2(\mu_T, \mu)]^{-(1 - q)} \cdot \E^\mu[\W_1(\mu_T, \mu)]^{2 - q}\\
 & \succeq \gg_d(T)^{-2 (1 - q)} \cdot \E^\mu[\W_1(\mu_T, \mu)]^{2 - q}, \ \ T \geq 2, \end{split}
 \end{equation*}
 where we also apply Theorem \ref{TN1} for $p = 1$ and $q = 2$. Therefore, we only need to prove \eqref{LB1} for $q = 1$. Recall that $\mu_{T, \eps}:=\mu_T \hat P_\vv= f_{T,\vv}\mu,$ where
 $f_{T, \eps} $ is defined in \eqref{FT}.
By  \eqref{CT2}, we have
$$  \W_1(\mu_T, \mu) \geq \kappa^{-1}  \W_1(\mu_{T, \eps}, \mu),\ \ T>0, \, \vv\in (0,1].$$
So,   it suffices to show  that for large $T>0$,
\beq\label{WT}
\sup_{\vv\in (0,1]}\E^\mu[\W_1 (\mu_{T, \eps}, \mu)] \succeq \gg_d(T).
\end{equation}
Let  $g_{T,\vv}:= (-\hat L)^{-1}(f_{T,\vv}-1).$
 By \eqref{LU}, for any $\xi > 0$, there exists a $\xi$-Lipschitz function $g_\xi$ such that
 $$E_\xi : = \{x\in M:   g_{T, \eps}(x) \ne g_\xi(x)   \}$$ satisfies
\begin{align} \label{except}
\mu (E_\xi) \leq \frac{C}{\xi^2} \int_M |\nabla g_{T, \eps}|^2 \d \mu.
\end{align}
Since $\|\nn g_\xi\|_\infty\le \xi$,
 Kantorovich's duality  implies
\begin{align*}
\W_1(\mu_{T, \eps}, \mu) \geq \xi^{-1} |\mu_{T,\vv}(g_\xi)-\mu(g_\xi)|= \xi^{-1} \abs{ \int_M g_\xi (f_{T, \eps} - 1) \d \mu} = \xi^{-1} \abs{ \hat{\EE}(g_\xi,   g_{T, \eps}) }.
\end{align*}
By \eqref{SA},  we have  $\mu$-a.e.  $1_{E_\xi^c}|\nabla (g_\xi -  g_{T, \eps})|=0$, so that
\beq\label{LN} \beg{split}
&\W_1(\mu_{T, \eps}, \mu) \geq  \xi^{-1} \Big|\hat{\EE} (g_{T,\vv}, g_{T,\vv})-  \hat{\EE}( g_{T, \eps}, g_{T,\vv}-g_\xi)  \Big|\\
&\ge\xi^{-1}\Big[ \hat{\EE} (g_{T,\vv}, g_{T,\vv})-   \mu\big(1_{E_{\xi} }|\nn  g_{T, \eps} |(|\nn g_{T,\vv}| + |\nn g_\xi|)\big)  \Big].\end{split}\end{equation}
Moreover, by $\|\nn g_\xi\|_\infty\le \xi$,  H\"older's inequality and \eqref{except}, we find a constant $c_1>0$ such that
\beg{align*} & \mu\big(1_{E_{\xi} }|\nn  g_{T, \eps} |(|\nn g_{T,\vv}| + |\nn g_\xi|)\big)  \le \mu(1_{E_\xi} |\nn g_{T,\vv}|^2)  + \xi \mu(1_{E_\xi} |\nn g_{T,\vv}|)\\
&\le \mu(E_\xi)^{\ff 1 4} \|\nn g_{T,\vv}\|_{L^4(\mu)} \|\nn g_{T,\vv}\|_{L^2(\mu)}  + \xi \mu(E_\xi)^{\ff 3 4} \|\nn g_{T,\vv}\|_{L^4(\mu)}\\
&\le c_1 \xi^{-\ff 1 2} \|\nn g_{T,\vv}\|_{L^4(\mu)} \|\nn g_{T,\vv}\|_{L^2(\mu)}^{\ff 3 2}. \end{align*}
This together with \eqref{LN} yields
$$\W_1(\mu_{T, \eps}, \mu) \ge \xi^{-1} \| \nn g_{T, \vv}\|_{L^2(\mu)}^2 - c_1 \xi^{-\ff 3 2} \|\nn g_{T,\vv}\|_{L^4(\mu)} \|\nn g_{T,\vv}\|_{L^2(\mu)}^{\ff 3 2}.$$
Taking expectation and using H\"older's inequality again, we obtain
\begin{equation} \label{f1} \beg{split}
&\E^\mu[\W_1(\mu_{T, \eps}, \mu)] \geq \xi^{-1} \E^\mu \big[ \| \nabla g_{T, \eps} \|_{L^2(\mu)}^2 \big] \\
&\qquad\qquad - c_1 \xi^{-\frac32} \big(\E^\mu \big[ \| \nabla g_{T, \eps} \|_{L^2(\mu)}^2 \big]\big)^\frac34 \big( \E^\mu \big[ \| \nabla g_{T, \eps} \|_{L^4(\mu)}^4 \big]\big)^\frac14.\end{split}
\end{equation}

(b) To estimate  $\E^\mu \big[ \| \nabla g_{T, \eps} \|_{L^2(\mu)}^2 \big]$, we formulate $g_{T,\vv}$ as
\begin{align*}
g_{T, \eps} = (- \hat L)^{-1} (f_{T, \eps} - 1) = \frac{1}{\sqrt{T}} \sum_{i = 1}^{\infty} \frac{\e^{-\eps \ll_i}}{\ll_i} \psi_i(T) \phi_i,
\end{align*} where $\psi_i(T)$ is defined in \eqref{PS}. Then
\begin{align} \label{*00}
\E^\mu \big[ \| \nabla g_{T, \eps} \|_{L^2(\mu)}^2 \big] = \ff{1}{T} \sum_{i = 1}^{\infty} \ff{\e^{- 2\ll_i \vv}}{\ll_i} \E^\mu[|\psi_i(T)|^2].
\end{align}
Using Lemma \ref{lem:ito-tanaka-eigen}, and choosing $m\ge 1$ such that $\ll_m > 4 c'^2$, we obtain
$$\E^\mu[|\psi_i(T)|^2] \ge \ff 2{\ll_i}\big(1 - c' \ll_i^{-\ff 1 2}\big)- \ff c{\ll_i T} \ge \ff 1 {\ll_i}- \ff{c}{\ll_i T}\ge \ff 1 {2\ll_i},\ \ T\ge 2c, \, i \ge m.$$
Hence combining this with \eqref{*00} and \eqref{EG2} yields
 \begin{align} \label{0}
\E^\mu \big[ \| \nabla g_{T, \eps} \|_{L^2(\mu)}^2 \big] \ge     \ff{1}{2T} \sum_{i = m}^{\infty} \ff{\e^{- 2\ll_i \vv}}{\ll_i^2}\ge   \ff{1}{2 \kappa^2 T} \sum_{i = m}^{\infty} \ff{\e^{- 2 \kappa i^{\ff  2 d} \vv}}{i^{\ff 4 d}},\ \ T\ge 2c.
\end{align}
By choosing
\begin{align} \label{epsilon}
\eps = \begin{cases} T^{-1}, & \textrm{if } d \in [1,4),\\
 T^{-\frac{2}{d - 2}}, &\textrm{if } d \in [4,\infty),\end{cases}
\end{align}
we derive that for large $T>0$,
\begin{align} \label{f2}
\E^\mu \big[ \| \nabla g_{T, \eps} \|_{L^2(\mu)}^2 \big]  \succeq \gg_d(T)^2=\begin{cases}
T^{-1}, &\textrm{if } d\in [1,4),\\
T^{-1} \log T , &\textrm{if } d = 4,\\
T^{-\frac{2}{d - 2}}, &\textrm{if } d\in(4,\infty).\end{cases}
\end{align}

(c) Noting that $g_{T,\vv} = (-\hat L)^{-1} (f_{T,\vv}-1)$,  we have
\beq\label{f22} \|\nn g_{T, \eps} \|_{L^2(\mu)} =\|\nn (-\hat L)^{-1} (f_{T,\vv}-1)\|_{L^2(\mu)}=\|(-\hat L)^{-\ff 1 2} (f_{T,\vv}-1)\|_{L^2(\mu)}.\end{equation}
Next, by \eqref{RZ} for $p=4$, we find a constant $c>0$ such that
\beq\label{f33} \| \nabla g_{T, \eps} \|_{L^4(\mu)}^4 \le c\| (-\hat L)^{-\ff 1 2}(f_{T, \eps}-1) \|_{L^4(\mu)}^4.\end{equation}
By combining \eqref{ED} with \eqref{f22} and \eqref{f33},    and using  the choice of $\eps$ in \eqref{epsilon},   we obtain that for large $T>0$,
\begin{align} \label{f3}
\big(\E^\mu \big[ \| \nabla g_{T, \eps} \|_{L^2(\mu)}^2 \big]\big)^\frac34 \big(\E^\mu \big[ \| \nabla g_{T, \eps} \|_{L^4(\mu)}^4 \big]\big)^\frac14 \preceq \gamma_d(T)^{\frac52} = \begin{cases}
T^{-\frac54}, &\textrm{if } d\in[1,4),\\
T^{-\frac54} (\log T)^{\frac54} , &\textrm{if } d = 4,\\
T^{-\frac{5}{2(d - 2)}}, &\textrm{if } d\in(4,\infty).\end{cases}
\end{align}
Combining \eqref{f1}, \eqref{f2} and \eqref{f3},  and choosing $\xi= K\gg_d(T)$
for a large enough constant  $K > 0,$ we obtain
\begin{equation*}   \E^\mu[\W_1 (\mu_T,\mu)]   \succeq \gg_d(T),\end{equation*}
 which implies \eqref{LB1} for   $q\ge 1$, and hence also for all $q\in (0,\infty)$ explained above.

We then aim to strengthen \eqref{LB1}   uniformly in  initial distribution. We first observe that  for some $s>0$,
\beq\label{CC0}    \inf_{x,y\in M}  p_s(x,y) \ge \ff 1 2.\end{equation}
According to \cite[Theorem 3.3.14]{Wbook}, \eqref{N} implies
$$\|P_t\|_{L^1(\mu)\to L^\infty(\mu)}<\infty,\ \ t>0.$$
Moreover, \eqref{N} also implies the Poincar\'e inequality
$$\mu(f^2)\le \ff 1 {\ll_1}   \hat \EE(f,f),\ \ \  f \in \D_0,$$
where $\ll_1>0$ is the spectral gap of $\hat L$. Since
$$\ff {\d}{\d t}\|P_t f\|_{L^2(\mu)}^2 = 2\< P_t f, L P_t f    \>_{L^2(\mu)} = -2\hat \EE(P_t f, P_t f),$$
the Poincar\'e inequality implies
$$\|P_t-\mu\|_{L^2(\mu)}\le \e^{-\ll_1 t},\ \ t\ge 0.$$
Therefore,  by the semigroup property, there exists a constant $c>0$ such that for any $t \geq 2$,
\beg{align*}&\sup_{x, y \in M} |p_t(x, y) - 1| = \|P_t-\mu\|_{L^1(\mu)\to L^\infty(\mu)} \\
& \le \|P_1-\mu\|_{L^1(\mu)\to L^2(\mu)} \|P_{t-2}-\mu\|_{L^2(\mu)} \|P_1 - \mu\|_{L^2(\mu)\to L^\infty(\mu)}
\le c\e^{-\ll_1 t}. \end{align*}
Thus, for large enough $s>0$ such that $c\e^{-\ll_1 s}\le \ff 1 2$, we have
$$\inf_{x, y \in M} p_s(x, y) \ge 1 - \sup_{x, y \in M} |p_t(x, y) - 1| \ge 1-\ff 1 2\ge \ff 1 2.$$
Hence, \eqref{CC0} holds.   Consequently, for any $\nu \in \scr P$, we have
\begin{equation} \label{CC0'}
\nu_s:= \nu P_s^* \ge \frac12 \mu.
\end{equation}

By \eqref{LB1}, \eqref{CC0'} and the Markov property, we deduce that
$\tt\mu_{T,s}$, which is given in \eqref{def of mu}, satisfies
\beq\label{AB1} \inf_{\nu\in \scr P} \E^{\nu}[\W_1^q(\tt\mu_{T,s},\mu)] = \inf_{\nu\in \scr P} \E^{\nu_s}[\W_1^q(\mu_T,\mu)] \ge \frac12 \E^\mu[\W_1^q(\mu_T,\mu)] \succeq \gg_d(T)^q \end{equation}  for large $T>0$.
On the other hand, \eqref{DPS} implies that for $D:=\sup\rr<\infty$,
$$\W_1(\mu_T,\tt\mu_{T,s})\le  \ff{ s D}T.$$
Combining this with \eqref{AB1} and   the triangle inequality, we conclude that for large $T>0$,
 \beg{align*} &\inf_{\nu\in \scr P} \E^\nu[\W_1^q(\mu_{T},\mu)] \ge  \inf_{\nu\in \scr P} \E^\nu\big[2^{-q} \W_1^q(\tt\mu_{T,s},\mu) -\W_1^q(\mu_T,\tt\mu_{T,s})\big]\\
& \succeq \gg_d(T)^q -  T^{-q} \asymp \gg_d(T)^q.\end{align*}
Then the proof is finished.

 \section{Proofs of  Theorem \ref{TN3}  and Theorem \ref{TN4}}

We first prove some lemmas.

\beg{lem}\label{LA1} Assume {\bf (A)} and let  $d\in [1, 4).$  Then
\begin{equation}\label{UP3'}
     \lim_{T\to\infty}  \E^\mu \Big[ \big| \{T\W_2^2(\mu_T,\mu) - \Xi(T)\}^+ \big|^q \Big]=0,\ \ q\in [1,\infty).\end{equation}\end{lem}
\beg{proof}

For $T \geq 2$ and $\vv \in (0, 1)$, let $\psi_i(T)$ and $f_{T, \vv}$ be defined  in  \eqref{HS2} and \eqref{FT} respectively. Then
\beq\label{XE}\beg{split}
&\Xi_\vv(T) := T \| \nn (-\hat L)^{-1} (f_{T,\vv}-1) \|_{L^2(\mu)}^2 = \sum_{i=1}^\infty \ff{|\psi_i(T)|^2 \e^{-2\vv \ll_i}}{\ll_i},\\
&\mu_{T, \eps, \eps} : = (1 - \eps) \mu_{T, \eps} + \eps \mu = [(1-\vv)f_{T,\vv}+\vv]\mu \in \scr P.
\end{split} \end{equation}
By \eqref{AM0}, we have
\beq\label{AM0'}  \W_2^2(\mu_{T,\vv,\vv}, \mu) \le \mu\Big(\big|\nn (-\hat L)^{-1} (f_{T,\vv}-1)\big|^2
\scr M\big((1-\vv)f_{T,\vv}+\vv\big)^{-1}\Big).\end{equation}

For $\eps \in (0,1)$, we define
$$
\mu_{T, \eps, \eps} : = (1 - \eps) \mu_{T, \eps} + \eps \mu \in \scr P,
$$
whose density function with respect to $\mu$ is given by $(1 - \eps) f_{T, \eps} + \eps$.

By \eqref{RZ} and  \eqref{ED}, we obtain for any $q \in [1, \infty)$, $ T \geq 2$ and $ \vv \in (0, 1),$
 \beg{align*} &T^q \E^\mu \big[ \| \nn (-\hat L)^{-1} (f_{T,\vv}-1) \|_{L^{2q}(\mu)}^{2q} \big] \les T^q\E^\mu \big[ \|  (-\hat L)^{-\frac12} (f_{T,\vv}-1) \|_{L^{2q}(\mu)}^{2q} \big]\\
& \les
\begin{cases} 1, & \textrm{if } d  \in [1, 3);\\
1 + T^{-q} \log^{2q}(1 + \vv^{-1}), & \textrm{if } d  = 3;\\
1 + T^{-q} \vv^{-(d - 3)q}, &\textrm{if } d \in (3, 4).  \end{cases}
\end{align*}
From now on, we choose
\beq\label{EPS} \eps = \eps(T) = T^{-\zeta},\end{equation}
for some $\zeta \in (1,  \frac{2}{(d - 2)^+})$, so that the above estimate reduces to
\begin{equation}\label{UPq}
\sup_{T\ge 2} T^q \E^\mu \big[ \| \nn (-\hat L)^{-1} (f_{T,\vv}-1) \|_{L^{2q}(\mu)}^{2q} \big] < \infty, \quad q \in [1, \infty).
 \end{equation}
On the other hand, by Lemma \ref{lem:ito-tanaka-eigen}, \eqref{IV} for $\alpha = 1$ and \eqref{EG},
we derive that for $T\ge 2$  and $ \vv \in (0, 1),$
\begin{equation}\label{L2}\begin{split}
&\E^\mu \big[ \| f_{T, \eps} - 1     \|_{L^2(\mu)}^2    \big] = \frac{1}{T} \sum_{i = 1}^{\infty} \e^{-2 \ll_i \eps} \E^\mu[|\psi_i(T)|^2] \les \frac{1}{T} \sum_{i = 1}^{\infty} \e^{-2 \ll_i \eps} \ll_i^{-1}\\
& \les \begin{cases} T^{-1}, & \textrm{if } d  \in [1, 2);\\
T^{-1} \log(1 +  \vv^{-1}), & \textrm{if } d  = 2;\\
T^{-1} \vv^{- \frac{d - 2}{2}}, &\textrm{if } d \in (2, 4). \end{cases}
\end{split}
\end{equation}
By the argument used in \cite[Lemma 3.3]{WZ}, for   $\eps$ in \eqref{EPS} we have
\begin{align} \label{log}
\lim_{T \to \infty} \E^\mu \big[ \mu \big( \big| \mathscr{M}\big( (1 - \eps)f_{T,\eps} + \eps \big)^{-1}-1 \big|^q    \big) \big] = 0, \quad q \in [1, \infty).
\end{align}
Indeed, for each $\eta \in (0, 1)$, if $|f_{T,\eps} - 1| \leq \eta$, then
$$
\big| \mathscr{M}\big( (1 - \eps)f_{T,\eps} + \eps \big)^{-1}-1 \big| \leq \left| \frac{1}{\ss{1 - \eta}} - \frac{2}{2 + \eta} \right| =: \delta(\eta);
$$
while if $|f_{T,\eps} - 1| > \eta$, for any $\theta \in (0, 1)$,
$$
\big| \mathscr{M}\big( (1 - \eps)f_{T,\eps} + \eps \big)^{-1}-1 \big| \leq 2 \theta^{-1} \eps^{-\frac{\theta}2}.
$$
Combining with Chebyshev's inequality yields
\begin{align*}
\E^\mu \big[ \mu \big( \big| \mathscr{M}\big( (1 - \eps)f_{T,\eps} + \eps \big)^{-1}-1 \big|^q    \big) \big] \leq \delta(\eta)^q + 2^q \eta^{-2} \theta^{-q} \eps^{-\frac{\theta q}2}  \E^\mu \big[ \| f_{T, \eps} - 1     \|_{L^2(\mu)}^2    \big].
\end{align*}
By \eqref{L2}, for our choice of $\eps$ and a sufficiently small $\theta$, we obtain
\begin{align*}
\limsup_{T \to \infty} \E^\mu \big[ \mu \big( \big| \mathscr{M}\big( (1 - \eps)f_{T,\eps} + \eps \big)^{-1}-1 \big|^q    \big) \big] \leq \delta(\eta)^q.
\end{align*}
Then \eqref{log} follows by letting $\eta \to 0$.

By \eqref{AM0'} and the H\"older's inequality, for any $q \in [1, \infty)$,
 \begin{equation*}\begin{split}
& \E^\mu\Big[ \big| \{T\W_2^2(\mu_{T,\eps, \eps},\mu)- (1 - \eps)^2 \Xi_\eps(T)\}^+ \big|^q\Big]\\
& \le (1 - \eps)^2 T^q \E^\mu\Big[ \mu \big( |\nabla(-\hat L)^{-1}(f_{T,\eps}-1)|^2 \cdot \big| \mathscr{M}\big( (1 - \eps)f_{T,\eps} + \eps \big)^{-1}-1 \big| \big)^q\Big]\\
&\le T^q \E^\mu\Big[ \mu \big( |\nabla(-\hat L)^{-1}(f_{T,\eps}-1)|^{2q} \cdot  \big| \mathscr{M}\big( (1 - \eps)f_{T,\eps} + \eps \big)^{-1}-1 \big|^q \big) \Big]\\
&\le \left( T^{2q} \E^\mu \big[ \| \nn (-\hat L)^{-1} (f_{T,\vv}-1) \|_{L^{4q}(\mu)}^{4q} \big]\right)^{\frac{1}{2}} \cdot \left( \E^\mu \big[ \mu \big( \big| \mathscr{M}\big( (1 - \eps)f_{T,\eps} + \eps \big)^{-1}-1 \big|^{2q}    \big) \big] \right)^{\frac{1}{2}}.
\end{split}
\end{equation*}
Combining this with \eqref{UPq}, \eqref{log}  and    $\Xi(T) \ge (1 - \eps)^2\Xi_\eps(T)$, we see  that for $\vv $ as in \eqref{EPS},
\begin{equation} \label{Pos}
\lim_{T\to\infty} \E^\mu\Big[ \big| \{T\W_2(\mu_{T,\eps, \eps},\mu)^2-\Xi(T)\}^+ \big|^q\Big] = 0,\ \ q \in [1,\infty).
\end{equation}
Moreover, by \eqref{CH} and the convexity of $\W_2^2$, we derive that for  $\eps$ as in \eqref{EPS},
\begin{align*}
& \lim_{T\to\infty}T^q\E^\mu \big[ \W_2^{2q}(\mu_T,\mu_{T,\eps, \eps}) \big] \les  \lim_{T\to\infty} T^q\E^\mu \big[\W_{2}^{2q}(\mu_T,\mu_{T,\eps}) + \W_{2}^{2q}(\mu_{T, \eps},\mu_{T,\eps, \eps}) \big]\\
& \les \lim_{T\to\infty} \big\{ T^q\E^\mu \big[ \W_{2q}^{2q}(\mu_T,\mu_{T,\eps}) \big] + T^q D^{2q} \eps^q \big\}  = 0,\ \ q\in [1,\infty).
\end{align*}
By the triangle inequality for $\W_2$,
\begin{align*}
& \E^\mu \big[ \big|\W_2^2(\mu_T, \mu) -   \W_2^2(\mu_{T, \eps, \eps}, \mu)    \big|^q   \big]\\
& \les \big(\E^\mu \big[ \W_2^{2q}(\mu_T, \mu) \big]\big)^\frac12 \big( \E^\mu \big[ \W_2^{2q}(\mu_T,\mu_{T,\eps, \eps}) \big] \big)^\frac12 + \E^\mu \big[ \W_2^{2q}(\mu_T,\mu_{T,\eps, \eps}) \big].
\end{align*}
Combining these with Theorem \ref{TN1}, we deduce that
$$
\lim_{T\to\infty} T^q \E^\mu \big[ \big|\W_2^2(\mu_T, \mu) -   \W_2^2(\mu_{T, \eps, \eps}, \mu)    \big|^q   \big] = 0, \ \ q\in [1,\infty),
$$
which together with \eqref{Pos} yields \eqref{UP3'}.
\end{proof}

\beg{lem} \label{control}
Assume \eqref{N} and either \eqref{RZ} or \eqref{GE-}, and let    $d\in [1, 4)$ and
$$q\in J(d):=  \beg{cases} [1,\infty), &\text{if}\ d\in [1,3],\\
[1,\ff{d-2}{2(d-3)}), &\text{if}\ d\in (3,4),\end{cases}$$
then
\begin{equation} \label{KZ}
\sum_{i=1}^\infty  \frac{ \sup_{T \geq 1} \big( \E^\mu \sqa{ |\psi_i(T)|^{2q} } \big)^{\frac1q}}{\ll_i} < \infty.
\end{equation}
Consequently,
\begin{equation} \label{82}
\sup_{T \geq 1, \, \eps \in (0, 1)} \E^\mu \sqa{ | \Xi_\eps(T)|^q } \leq \sup_{T \geq 1} \E^\mu \sqa{ | \Xi(T)|^q} < \infty,
\end{equation}
\begin{equation}  \label{512}
\lim_{T \to\infty} \E^\mu \big[ | \Xi_\eps(T) - \Xi(T)|^q  \big] = 0
\end{equation}
holds for any $\eps = \eps(T) \in (0, 1)$ satisfying $\lim_{T \to \infty} \eps(T) = 0$.
\end{lem}

\begin{proof}
(a) Let $\si(\cdot)$ and $\mathfrak{m}(\cdot)$ be defined in \eqref{SM}.
It is easy to see that for any $i\ge 1$,
$$\si(\phi_i)= \ss{2} \ll_i^{-\ff 1 2}.$$
By $\|\phi_i\|_{L^2(\mu)}=1$, \eqref{HS} and $s=\ll_i^{-1}$,  we obtain that for any $p \in [2, \infty]$,
\beq \label{EG3}
\| \phi_i \|_{L^p(\mu)} =\e^{\ll_i s}\|\hat P_s \phi_i\|_{L^p(\mu)}\les \ll_i^{\ff{d(p-2)}{4p}}, \ \ i \geq 1.
\end{equation}

If \eqref{GE-} holds, then by \eqref{HS} we find  constants $c_1,c_2>0$  such that for any $s \in (0, \ll_1^{-1})$ and $p \in (2, \infty)$,
\begin{equation}\label{EG4} \begin{split}
&\|\nabla \hat{P}_s\phi_i\|_{L^p(\mu)}\le c_1\|\hat{P}_s(|\nabla \phi_i|^2)\|_{L^{\frac p 2}(\mu)}^{\frac 1 2} \\
&\le c_1 \|\hat{P}_s\|_{L^1(\mu)\to L^{\frac p 2}(\mu)}^{\frac 1 2}\|\nabla \phi_i\|_{L^2(\mu)}\le c_2 s^{-\frac{(p-2)d}{4p}}\lambda_i^{\frac 1 2}, \ \ i \geq 1.
\end{split}\end{equation}
By letting $s=\lambda_i^{-1}$, we obtain that for any $p \in (2, \infty)$,
$$\|\nabla \phi_i\|_{L^p(\mu)} = \e^{\ll_i s} \| \nabla \hat P_s \phi_i\|_{L^p(\mu)} \les \lambda_i^{\frac{d(p-2)}{4p} + \frac12}.$$
So, by   taking $p>2$ such that $\ff p{p-2}=  \log (\e+\ll_i)$ for $d=2$ and $p = d$ for $d>2$, we deduce from \eqref{SM} that
 \beq\label{EG*}  \mathfrak{m}(\phi_i)
 \les \beg{cases} \ll_i^{-\ff 1 2},\ &\text{if}\  d\in [1,2),\\
  \ll_i^{-\ff 1 2} \log(\e+\ll_i),\ &\text{if}\  d=2,\\
\ll_i^{-\ff {4-d}4 },\ &\text{if}\  d\in (2, \infty). \end{cases} \end{equation}
If   \eqref{RZ} holds, then we deduce \eqref{EG*} from \eqref{RZ} and \eqref{EG3}. So, in both cases we have estimate \eqref{EG*}.

Now, by combining \eqref{EG*} with  Corollary \ref{C3.2}, we derived that for any $q\in [1,\infty)$ and $i\ge 1$,
\begin{align} \label{psiT}
\sup_{T \geq 1}\big( \E^\mu \sqa{ |\psi_i(T)|^q }\big)^{\frac1q} \les   \begin{cases} \ll_i^{- \frac12}, & \textrm{if } d  \in [1,2),\\
\ll_i^{- \frac12} \log(\e+\ll_i), & \textrm{if } d  = 2,\\
\ll_i^{-\frac{4-d}4}, &\textrm{if } d\in (2, \infty). \end{cases}
\end{align}
This together with  \eqref{EG} implies  \eqref{KZ}   for $d \in [1, 2]$ and all $q\in [1,\infty)$.

(b) Next, we verify \eqref{KZ} for $d\in (2 , 4)$.
By Lemma \ref{lem:ito-tanaka-eigen},
\beq\label{TT1}
\sup_{T \geq 1} \E^\mu \sqa{ |\psi_i(T)|^2 } \les \ll_i^{-1},\ \ i\ge 1,\end{equation}
 while  H\"older's inequality implies that for any $\tau \in (0,1)$,
$$    \sup_{T \geq 1}\big( \E^\mu \sqa{ |\psi_i(T)|^{2q} }\big)^{\frac1q}
 \le
\sup_{T \geq 1}\bigg( \E^\mu \sqa{ |\psi_i(T)|^{2 } }\bigg)^{\frac {1 - \tau} q} \cdot \sup_{T \geq 1} \bigg( \E^\mu \sqa{ |\psi_i(T)|^{\ff{2(q - 1 + \tau)}{\tau} }}\bigg)^{\frac {\tau} q}.$$
Since $\tau \in (0, 1)$ is arbitrary,  this together with  \eqref{psiT} and  \eqref{TT1} yields  that for any $q\in [1,\infty)$ and   $\varsigma \in (0,1)$,
$$  \sup_{T \geq 1}\big( \E^\mu \sqa{ |\psi_i(T)|^{2q} }\big)^{\frac1q}  \les \ll_i^{-\frac{4 - d}{2} - \frac{d - 2}{2q} + \varsigma}, \ \ i\ge 1.$$
    Consequently,
\beq  \label{TJ'}
\sum_{i=1}^\infty  \frac{ \sup_{T \geq 1} \big( \E^\mu \sqa{ |\psi_i(T)|^{2q} } \big)^{\frac1q}}{\ll_i} \les \sum_{i=1}^\infty \ll_i^{-1-\frac{4 - d}{2} - \frac{d - 2}{2q} +  \varsigma}.
\end{equation}
By \eqref{EG}, the series $\sum_{i = 1}^{\infty} \ll_i^{-\theta}$ converges when $\theta > \frac{d}{2}$.
 It is easy to see that
$$ 1+\frac{4 - d}{2} + \frac{d - 2}{2q} - \varsigma > \ff d 2 \ \text{holds\  for\  small}\   \varsigma \in (0,1)$$ if and only if
$q\in [1, \ff{d-2}{2(d-3)^+})$, where $\ff{d-2}{2(d-3)^+}:=\infty$ for $d\in (2,3].$
So,  by  \eqref{TJ'}, \eqref{KZ} holds for any $d \in (2, 4)$ and  $q\in J(d)=[1, \ff{d-2}{2(d-3)^+})$.

(c) Let $\Xi(T)$ be in \eqref{PS}. By the Minkowski  integral inequality, for $q \in [1, \infty)$,
\begin{align*}
\big( \E^\mu \sqa{ | \Xi(T)|^q} \big)^\frac{1}{q} \leq \sum_{i=1}^\infty  \frac{ \big( \E^\mu \sqa{ |\psi_i(T)|^{2q} } \big)^{\frac1q}}{\ll_i}.
\end{align*}
Then \eqref{82} follows from \eqref{KZ} and the fact that $\Xi_\eps(T) \leq \Xi(T)$. Noting  that
\beq\label{511'}
\big(\E^\mu [  | \Xi_\eps(T) - \Xi(T)|^q  ]\big)^{\frac1q}   \leq  \sum_{i=1}^\infty \frac{(\E^\mu \sqa{ |\psi_i(T)|^{2q} })^{\frac1q}}{\ll_i} (1 - \e^{-2\ll_i \eps}),
\end{equation}
by the dominated convergence theorem,  we deduce \eqref{512} from \eqref{511'} and \eqref{82}. Then the proof is finished.
\end{proof}

For a general initial distribution, we have the following result, where   Lemma \ref{lem:ito-tanaka-eigen} indicates  that the   estimate \eqref{psiT'} is nearly sharp for   $d \in [1, 2]$.

\begin{lem} \label{control2}
Assume {\bf (A)} and let $d\in [1, \infty).$ Then for any $q \in [1, \infty)$,
\begin{equation} \label{psiT'}
\sup_{T\ge 2, \, \nu \in\mathscr{P}} \big( \E^\nu \sqa{ |\psi_i(T)|^q }\big)^{\frac1{q}} \les \ll_i^{-\frac{2 \wedge (4- d)}{4} + \varsigma}, \ \ i \geq 1,
\end{equation}
where $\varsigma > 0$ is an arbitrarily small but fixed constant. Moreover, for any   $\nu \in \mathscr{P}$,
\begin{equation} \label{NORM}
\lim_{T \to \infty} \E^\nu \sqa{ |\psi_i(T)|^{q}} = \frac{2^q \Gamma(\frac{q + 1}2)}{\sqrt{\pi}} \cdot {\bf V}(\phi_i)^{\frac{q}2}, \ \ i \geq 1, \, q\in (0,\infty),
\end{equation}
where ${\bf V}( \cdot )$ is defined by \eqref{def of V}.
\end{lem}

\begin{proof}
(a) For $\nu \in \scr P$, let $\nu_\eps = \nu P_{\eps}^*$. Then by the Markov property, H\"older's inequality and \eqref{A0'}, there exists a constant $c_1 > 0$ such that for any $\eps \in (0, 1)$ and $k \in (1, \infty)$,
\begin{equation*} \begin{split}
& \sup_{T \ge 2,\,\nu \in \scr P}\big(  \E^\nu \big[  |\psi_i(T)|^q \big] \big)^{\frac1q} \leq \sup_{T \geq 2, \, \nu \in \scr P} \E^\nu \left[  \bigg| \int_{0}^{\eps} \phi_i(X_t) \d t \bigg|^q \right]^{\frac1{q}} + \sup_{T \geq 1, \, \nu \in \scr P} \big( \E^{\nu_\eps} \big[  |\psi_i(T)|^{q} \big]\big)^{\frac1{q}}\\
& \leq \| \phi_i \|_\infty \eps + c_1 \eps^{-\frac{d(k - 1)}{2kq}}    \sup_{T \geq 1} \big( \E^{\mu} \big[  |\psi_i(T)|^{\frac{kq}{k - 1}} \big] \big)^{\frac{k - 1}{kq}}, \ \ i \geq 1.
\end{split}
\end{equation*}
Combining this with \eqref{EG3} for $p = \infty$ and \eqref{psiT} yields
 \begin{equation*}
\sup_{T \ge 2,\,\nu \in \scr P}\big(  \E^\nu \big[  |\psi_i(T)|^{q} \big] \big)^{\frac1{q}} \les \ll_i^{\frac{d}4} \eps + \eps^{-\frac{d(k - 1)}{2kq}} \ll_i^{-\frac{2 \wedge (4- d)}{4}}\big( \log(\e+\ll_i) 1_{\{ d = 2 \}} + 1_{\{ d \ne 2 \}}  \big), \ \ i \geq 1.
\end{equation*}
By choosing $k$ sufficiently close to $1$ and optimizing in $\eps \in (0, 1)$, we prove \eqref{psiT'}.

(b) Given $q \in [1, \infty)$, by \eqref{psiT'} and Chebyshev's inequality, there exists a constant $C > 0$ depending on $i$, such that for any $r > 0$,
\begin{align} \label{tail}
\sup_{T \ge 2, \, \nu \in \mathscr P} \mathbb P^\nu (|\psi_i(T)| \geq r) \leq C r^{-q - 1}.
\end{align}

Now, we introduce a continuous indicator function $\chi \in C(\R)$ which satisfies $0 \leq \chi \leq 1$, $\chi \equiv 1$ on $[-1, 1]$ and $\text{supp }\chi \subseteq [-2, 2]$. By \cite[Theorem 2.1]{Wu2}, for any $K > 0$,
\begin{align*}
\lim_{T \to \infty} \E^\nu \big[  |\psi_i(T)|^{q} \chi(K^{-1} \psi_i(T)) \big] = \int_\R |x|^{q} \chi(K^{-1} x) \, N(0, 2 {\bf V}(\phi_i))(\d x).
\end{align*}
On the other hand, by \eqref{tail}, for any $T \geq 2$,
\begin{align*}
0 \leq &\E^\nu \big[  |\psi_i(T)|^{q} \big] - \E^\nu \big[  |\psi_i(T)|^{q} \chi(K^{-1} \psi_i(T)) \big] \leq \E^\nu \big[  |\psi_i(T)|^{q} 1_{\{|\psi_i(T)| > K \}} \big]\\
& \leq K^{q} \mathbb P^\nu (|\psi_i(T)| \geq K) +  q \int_{K}^{\infty} r^{q - 1} \mathbb P^\nu (|\psi_i(T)| \geq r) \, \d r \leq (q + 1) C K^{-1}.
\end{align*}
Letting $K \to \infty$, we obtain
\begin{align*}
\lim_{T \to \infty} \E^\nu \big[  |\psi_i(T)|^{q} \big] = \int_\R |x|^{q} \, N(0, 2 {\bf V}(\phi_i))(\d x),
\end{align*}
which finishes the proof of \eqref{NORM}.
\end{proof}

\beg{lem}\label{LA2} Assume {\bf (A)} and let $d\in [1,4)$. Then
\beq\label{LMO1}
\lim_{\eps\to 0}\sup_{T\ge 1, \, \nu\in\mathscr{P}}\E^{\nu}\big[|T\W_2^2(\mu_T,\mu)-T\W_2^2(\tt\mu_{T,\eps},\mu)|^{q}
\big]=0,\ \ q\in [1,\infty).\end{equation}
Moreover, if $d\in [1, \frac83)$  and
$q \in  [1, \ff{d}{2(d-2)^+})$,  where $\ff{d}{2(d-2)^+} :=\infty$ for $d \le 2$,
then
\beq\label{LMO2}
\lim_{\eps\to 0}\sup_{T\ge 2,\, \nu\in\mathscr{P}}\E^{\nu}\big[|\Xi(T)-\tt\Xi_\vv(T)|^q\big]=0.\end{equation}
\end{lem}

\beg{proof}
(a) Recall that
$$\tt\mu_{T,\eps} = \frac 1 T \int_{\eps}^{T + \eps}\delta_{X_t}\d t, \ \ \eps\in(0,1), \, T \geq 1,$$
and for $\nu \in \scr P$, let $\nu_\eps = \nu P_{\eps}^*$. By the Markov property, Theorem \ref{TN1} and \eqref{A0'}, for any $q\in [1,\infty)$ and $k \in (1, \infty)$, we find constants $c_1, c_2>0$ such that
\begin{align*}
&\sup_{\nu \in \scr P} \E^{\nu} \big[ \W_2^{2q}(\tt\mu_{T,\eps},\mu) \big] = \sup_{\nu \in \scr P} \E^{\nu_{\eps}} \big[ \W_2^{2q}(\mu_T,\mu) \big] \\
&\le c_1 \eps^{-\frac{d(k - 1)}{2k}} \E^{\mu}\big[ \W_2^{\frac{2kq}{k - 1}}(\mu_T,\mu) \big]^{\frac{k-1}{k}}\le c_2  \eps^{-\frac{d(k - 1)}{2k}} T^{-q},\ \ T\ge 1, \, \vv\in (0,1).
\end{align*}
By combining this with  the triangle inequality for $\W_2$,  \eqref{DPS} for $p=2$, and
$D:=\sup\rr<\infty$ due to the boundedness of $(M,\rr)$, we obtain
\begin{align*}
& \sup_{T\ge 1,\, \nu\in\mathscr{P}}\E^\nu \big[ |T\W_2^2(\mu_T,\mu) - T\W_2^2(\tt\mu_{T,\eps},\mu)|^{2q} \big]\\
& \le \sup_{T\ge 1,\, \nu\in\mathscr{P}} \E^\nu \big[ |T\W_2^2(\mu_T,\tt\mu_{T,\eps}) + 2T\W_2(\mu_T,\tt\mu_{T,\eps})\W_2(\tt\mu_{T,\eps},\mu)|^{2q} \big]\\
& \les D^{4q} \vv^{2q}+ c_2D^{2q} \vv^{ q -\frac{d(k - 1)}{2k}}.
\end{align*}
 By choosing sufficiently small $k$ such that $q -\frac{d(k - 1)}{2k} > 0,$  and then letting $\eps \to 0$, we derive   \eqref{LMO1}.

(b) Let $(\Xi(T),\psi_i(T))$ and $(\tilde{\Xi}_\eps(T), \tilde{\psi}_{i, \eps}(T))$ be in \eqref{PS} and \eqref{NNM} respectively. Then by Minkowski's  integral inequality and H\"older's inequality, we obtain that for any $\nu \in \scr P$ and $\tau \in (0, 1)$,
\begin{align*}
&\big( \E^\nu \big[|\Xi(T)- \tilde{\Xi}_\eps(T) |^q \big] \big)^{\frac1q} \leq \sum_{i = 1}^{\infty} \frac{1}{\ll_i} \E^\nu \left[ \big| |\psi_i(T)|^2 - |    \tilde{\psi}_{i, \eps}(T)|^2 \big|^q     \right]^{\frac1q}\\
 &   \leq \sum_{i = 1}^{\infty} \frac{1}{\ll_i} \E^\nu \left[ |\psi_i(T) -  \tilde{\psi}_{i, \eps}(T)|^q \cdot \big( |\psi_i(T)| + |\tilde{\psi}_{i, \eps}(T)| \big)^q   \right]^{\frac1q}\\
& \leq \sum_{i = 1}^{\infty} \frac{1}{\ll_i} \E^\nu \big[ |\psi_i(T) -  \tilde{\psi}_{i, \eps}(T)|^{\frac{q}{\tau}} \big]^{\frac{\tau}{q}} \cdot \left( \E^\nu \big[  |\psi_i(T)|^{\frac{q}{1 - \tau}} \big]^{\frac{1 - \tau}{q}} +  \E^\nu \big[  |\tilde{\psi}_{i, \eps}(T)|^{\frac{q}{1 - \tau}} \big]^{\frac{1 - \tau}{q}}  \right).
\end{align*}
A direct computation gives
\begin{equation} \label{term1} \begin{split}
& \sup_{T \ge 1,\, \nu \in \scr P} \E^\nu \big[ |\psi_i(T) -  \tilde{\psi}_{i, \eps}(T)|^{\frac{q}{\tau}} \big]^{\frac{\tau}{q}} \leq \sup_{T \ge 1,\, \nu \in \scr P} \E^\nu \left[ \bigg| \int_{0}^{\eps} \big( \phi_i(X_t) - \phi_i(X_{T + \eps}) \big) \d t \bigg|^{\frac{q}{\tau}} \right]^{\frac{\tau}{q}}\\
& \leq 2 \sup_{\nu \in \scr P} \E^\nu \left[  \bigg| \int_{0}^{\eps} \phi_i(X_t) \d t \bigg|^{\frac{q}{\tau}} \right]^{\frac{\tau}{q}} \leq 2  \int_{0}^{\eps} \sup_{\nu \in \scr P} \E^\nu\big[ |\phi_i(X_t)|^{\frac{q}{\tau}} \big]^{\frac{\tau}{q}}   \d t\\
& \leq 2 \int_{0}^{\eps} \big\| P_t(|\phi_i|^{\frac{q}{\tau}}) \big\|_\infty^{\frac{\tau}{q}} \d t.
\end{split}
\end{equation}
By \eqref{HS3} and \eqref{EG3} for $p = 2 \vee \frac{q}{\tau}$, we find a constant $c > 0$ such that
\begin{align*}
\big\| P_t(|\phi_i|^{\frac{q}{\tau}}) \big\|_\infty^{\frac{\tau}{q}} \leq \| P_t \|_{L^1(\mu) \to L^{\infty}(\mu)}^{\frac{\tau}{q}} \cdot \| \phi_i \|_{L^{\frac{q}{\tau}}(\mu)} \leq c t^{-\frac{d \tau}{2q}} \ll_i^{\frac{d(q - 2 \tau)^+}{4q}},\ \ t\in (0,1).
\end{align*}
Combining this with \eqref{term1} yields for any $\tau \in (0, 1 \wedge \frac{2q}{d})$,
\begin{equation*}
\sup_{T \ge 1,\, \nu \in \scr P} \E^\nu \big[ |\psi_i(T) -  \tilde{\psi}_{i, \eps}(T)|^{\frac{q}{\tau}} \big]^{\frac{\tau}{q}} \les \eps^{1 - \frac{d \tau}{2q}} \ll_i^{\frac{d(q - 2 \tau)^+}{4q}}, \ \ \eps \in (0, 1),  \,  i \geq 1.
\end{equation*}

On the other hand, we take $k \in (1, \infty)$. By the Markov property, H\"older's inequality, \eqref{A0'} and \eqref{psiT}, we obtain that for $i\ge 1$,
\begin{equation*} \begin{split}
& \sup_{T \ge 1,\, \nu \in \scr P} \E^\nu \big[  |\tilde{\psi}_{i, \eps}(T)|^{\frac{q}{1 - \tau}} \big]^{\frac{1 - \tau}{q}} = \sup_{T \ge 1, \, \nu_\eps \in \scr P} \E^{\nu_\eps} \big[  |\psi_i(T)|^{\frac{q}{1 - \tau}} \big]^{\frac{1 - \tau}{q}}\\
 &\leq \sup_{T \geq 1, \,\nu\in \scr P} \Big\| \ff{\d\nu_\eps}{\d\mu}\Big\|_{L^k(\mu)} \cdot \E^{\mu} \big[  |\psi_i(T)|^{\frac{kq}{(k - 1)(1 - \tau)}} \big]^{\frac{(k - 1)(1 - \tau)}{kq}}\\
& \les  \eps^{-\frac{d(k - 1)}{2k}} \ll_i^{-\frac{2 \wedge (4- d)}{4}+\varsigma },\ \  \ \ i\ge 1,
\end{split}
\end{equation*}
for any fixed constant  $\varsigma > 0$. Moreover,  Lemma \ref{control2} yields
\begin{equation*}
\sup_{T \ge 2,\, \nu \in \scr P} \E^\nu \big[  |\psi_i(T)|^{\frac{q}{1 - \tau}} \big]^{\frac{1 - \tau}{q}} \les \ll_i^{-\frac{2 \wedge (4- d)}{4} + \varsigma},\ \ \ \ i\ge 1.
\end{equation*}
Combining the above estimates, we conclude that for any $\tau \in (0, 1 \wedge \frac{2q}{d})$, $k \in (1, \infty)$ and $\varsigma \in (0, 1)$, there exists a constant $C_{\tau,k,\varsigma} \in (0,\infty)$ such that
\beq\label{GH}
\sup_{T \geq 2, \, \nu \in \scr P} \big( \E^\nu \big[|\Xi(T)- \tilde{\Xi}_\eps(T) |^q \big] \big)^{\frac1q} \le C_{\tau,k,\varsigma}\,  \eps^{1 - \frac{d \tau}{2q} - \frac{d(k - 1)}{2k}} \sum_{i = 1}^{\infty} \ll_i^{-1 + \frac{d(q - 2 \tau)^+}{4q} -\frac{2 \wedge (4- d)}{4} + \varsigma}.
\end{equation}

Noting that when $d\in [1, \frac{8}3)$  and $q \in [1, \ff{d}{2(d-2)^+})$, where $\ff{d}{2(d-2)^+} := \infty$ for $d \le 2$,  we have
$$
\lim_{\tau\uparrow 1\land \ff{2q}{d},\ \varsigma\downarrow 0} \Big[-1 + \frac{d(q - 2 \tau)^+}{4q} -\frac{2 \wedge (4- d)}{4} + \varsigma\Big]= -1 + \ff{d ( q - 2 )^+}{4q} -\ff{2 \land(4-d)}4 <-\ff d 2,
$$
and
$$
\lim_{k\downarrow 1} \Big[1 - \frac{d \tau}{2q} - \frac{d(k - 1)}{2k}\Big] = 1-\ff{d\tau}{2q}>0, \ \  \tau\in \Big(0,  1 \wedge \frac{2q}{d} \Big).
$$
Then we can always find $\tau \in (0, 1 \wedge \frac{2q}{d})$, $k \in (1, \infty)$ and $\varsigma \in (0, 1)$ such that
\beg{align*}
-1 + \frac{d(q - 2 \tau)}{4q} -\frac{2 \wedge (4- d)}{4} + \varsigma  <-\ff d 2 \ \text{ and } \
 1 - \frac{d \tau}{2q} - \frac{d(k - 1)}{2k} >0
 \end{align*}
hold. Hence, \eqref{EG} implies
$$
\sum_{i = 1}^{\infty} \ll_i^{-1 + \frac{d(q - 2 \tau)^+}{4q} -\frac{2 \wedge (4- d)}{4} + \varsigma} < \infty,
$$
so that by letting $\vv\to 0$ in \eqref{GH} we  deduce  \eqref{LMO2}.
\end{proof}

\beg{proof}[Proof of Theorem \ref{TN3}(1)] Let $\nu= h_\nu\mu\in \scr P_{k,R}.$ We have
$\E^\nu[\eta]=\E^\mu[h_\nu(X_0)\eta]$ for a nonnegative measurable functional $\eta$ of $(X_\cdot)$.
Let $k^*=\ff k{k-1}.$ By  H\"older's inequality  we obtain
\beg{align*} &\E^\nu\big[|\{T\W_2^2(\mu_T,\mu)-\Xi(T)\}^+|^q\big]= \E^\mu\big[ h_\nu(X_0) |\{T\W_2^2(\mu_T,\mu)-\Xi(T)\}^+|^q\big] \\
&\le \|h_\nu\|_{L^k(\mu)} \Big(\E^\mu\big[ |\{T\W_2^2(\mu_T,\mu)-\Xi(T)\}^+|^{k^* q}\big]\Big)^{\frac{1}{k^*}}\\
&\le R \Big(\E^\mu\big[  |\{T\W_2^2(\mu_T,\mu)-\Xi(T)\}^+|^{k^* q}\big]
\Big)^{\frac{1}{k^*}}.\end{align*}
  Combining this with Lemma \ref{LA1} and letting $T \to \infty$,  we derive \eqref{UP3}.

Next, let $d \in [1, \frac83)$ and $q\in [1,\ff d{2(d-2)^+})$. For $\nu \in \scr P$, let $\nu_\eps : = \nu P_\eps^*$ with $\eps \in (0, 1)$. By the fact that $\nu_\vv\le c_1\vv^{-\ff d 2} \mu$, \eqref{UP3} for $k = \infty$,  and the Markov property, we obtain
\beg{align*}&\lim_{T\to\infty}\sup_{\nu\in \scr P} \E^\nu\big[|\{T\W_2^2(\tt\mu_{T,\vv},\mu)-\tt\Xi_\vv(T)\}^+|^q\big]
=\lim_{T\to\infty}\sup_{\nu\in \scr P} \E^{\nu_\vv}\big[|\{T\W_2^2( \mu_{T},\mu)- \Xi(T)\}^+|^q\big]=0.\end{align*}
Then by triangle inequality, for any $\vv\in (0,1)$,
\beg{align*} &\limsup_{T\to\infty}\sup_{\nu\in \scr P} \E^{\nu} \big[|\{T\W_2^2( \mu_{T},\mu)- \Xi(T)\}^+|^q\big]\\
&\les \limsup_{T\to\infty}\sup_{\nu\in \scr P} \E^{\nu}\big[|\{T\W_2^2( \tt\mu_{T,\vv},\mu)- \tt\Xi_\vv(T)\}^+|^q\big]\\
& \qquad  +  \sup_{T\ge 2, \nu\in \scr P} \E^\nu\big[|T\W_2^2(\mu_T,\mu)-T\W_2^2( \tt\mu_{T,\vv},\mu)|^q+|\Xi(T)-\tt\Xi_\vv(T)|^q\big]\\
&= \sup_{T\ge 2, \nu\in \scr P} \E^\nu\big[|T\W_2^2(\mu_T,\mu)-T\W_2^2( \tt\mu_{T,\vv},\mu)|^q+|\Xi(T)-\tt\Xi_\vv(T)|^q\big].\end{align*}
By combining   this inequality  with Lemma \ref{LA2} and letting  $\vv\to 0$,  we derive  \eqref{TUP3}.
\end{proof}

\begin{proof}[Proof of  Theorem \ref{TN3}(2)]   Assume {\bf (A)} and {\bf (C)}.

  (a)  As explained in the proof of Theorem \ref{TN3}(1) that, the latter assertion for $d \in [1, \frac83)$ and $q\in [1,\ff d{2(d-2)^+})$ is a consequence of the former assertion for $k = \infty$ and  Lemma \ref{LA2}. Moreover, the argument used in the proof of \eqref{UP3} shows that \eqref{LM} is derived from
\begin{align} \label{tag0}
 \lim_{T\to\infty}  \E^\mu\Big[ \big| \{T\W_2^2(\mu_T,\mu) - \Xi(T)\}^- \big|^q\Big]=0,\ \ q \in J(d).
\end{align}
So, it remains to prove this formula.

(b) By the triangle inequality and Minkowski's inequality, for $\eps \in (0, 1)$,
\begin{equation}\label{WM}
\begin{split}
& \E^\mu\Big[ \big| \{T\W_2^2(\mu_T,\mu) - \Xi(T)\}^- \big|^q\Big]^{\frac1q} \\
& \leq \E^\mu\Big[ \big| \{T\W_2^2(\mu_T,\mu) - T\W_2^2(\mu_{T, \eps}, \mu)\}^- \big|^q\Big]^{\frac1q} + \E^\mu\Big[ \big| \{T\W_2^2(\mu_{T, \eps},\mu) - \Xi_\eps(T)\}^- \big|^q\Big]^{\frac1q} \\
& \qquad + \E^\mu  [ | \Xi_\eps(T) - \Xi(T)|^q  ]^{\frac1q}.
\end{split}
\end{equation}
Since $\mu_{T, \eps} = \mu \hat P_\eps$, it follows from \eqref{B22} that
\begin{align*}
\E^\mu\Big[ \big| \{T\W_2^2(\mu_T,\mu) - T\W_2^2(\mu_{T, \eps}, \mu)\}^- \big|^q\Big]^{\frac1q} \leq |h(\eps) - 1| \cdot T \E^\mu \big[  \W_2^{2q}(\mu_T,\mu)  \big]^{\frac1q}.
\end{align*}
Combining this  with Theorem \ref{TN1}, we derive
\begin{align*}
\lim_{T \to \infty} \E^\mu\Big[ \big| \{T\W_2^2(\mu_T,\mu) - T\W_2^2(\mu_{T, \eps}, \mu)\}^- \big|^q\Big] = 0
\end{align*}
for any $\eps = \eps(T) \in (0, 1)$ satisfying $\lim_{T \to \infty} \eps(T) = 0$. Combining this with  \eqref{512} and \eqref{WM}, to prove \eqref{tag0}, it suffices to prove that
\begin{align}\label{tag1}
\lim_{T\to\infty}  \E^\mu\Big[ \big| \{T\W_2^2(\mu_{T, \eps},\mu) - \Xi_\eps(T)\}^- \big|^q\Big]=0,\ \ q\in J(d),
\end{align}
holds for some choice of $\eps = \eps(T) \in (0, 1)$ with $\lim_{T \to \infty} \eps(T) = 0$.

(c) For any $q \in J(d)$, we take a $\varsigma > 0$ sufficiently small such that $q + \varsigma \in J(d)$. Then by \eqref{82} and Chebyshev's inequality, there exists a constant $C > 0$ depending on $q + \varsigma$ such that
\begin{align*}
\P^\mu \big(  |\Xi_\eps(T)| > r  \big) \leq \P^\mu \big(  |\Xi (T)| > r  \big) \leq 1 \wedge ( C r^{-q - \varsigma}), \ \ r > 0, \, T\ge 1, \, \vv\in (0,1).
\end{align*}
Hence, according to the dominated convergence theorem on the measure space $(\R_+, r^{q-1} \d r)$, in order to prove \eqref{tag1}, it is enough to show that for some proper choice of $\eps = \eps(T) \in (0, 1)$ with $\lim_{T \to \infty} \eps(T) = 0$,
\beq\label{81}
\lim_{T\to\infty}\P^\mu(\{ T\W_2^2(\mu_{T, \eps},\mu)-\Xi_\eps(T)\}^->r\big) = 0,  \ \ r > 0.
\end{equation}

By the argument used the proof of \cite[Theorem 2.4, p.35-37]{Wang23NS}, under {\bf (C)} we have
\beq\label{WM2}
\limsup_{T\to\infty}\P^\mu(\{ T\W_2^2(\mu_{T, \eps},\mu)-\Xi_\eps(T)\}^->r\big) \leq \limsup_{T \to \infty} \P^\mu (B_{T, \eps}^c), \ \ \eps \in (0, 1),
\end{equation}
where $\theta$ is the constant in the assumption {\bf (C)}, and
$$
B_{T, \eps} : = \big\{  \| f_{T, \eps} - 1 \|_\infty^2 \leq c_1 T^{-\frac{2 + 2 \theta}{2 + 3\theta}}   \big\}
$$
for some constant $c_1 > 0$. By Chebyshev's inequality and \eqref{HS} for $(2, \infty)$ in place of $(p, q)$, we find a constant $c_2 >0$ such that
\begin{align*}
\P^\mu (B_{T, \eps}^c) \leq c_1^{-1} T^{\frac{2 + 2 \theta}{2 + 3\theta}} \E^\mu \big[ \| f_{T, \eps} - 1     \|_\infty^2    \big] \leq c_2 T^{\frac{2 + 2 \theta}{2 + 3\theta}} \eps^{-\frac{d}{2}} \E^\mu \big[ \| f_{T, \frac{\vv}{2}} - 1     \|_{L^2(\mu)}^2    \big].
\end{align*}
Combining this estimate with \eqref{L2}, it follows that for the choice $\eps = T^{-\zeta}$ with sufficiently small $\zeta>0$,
\begin{align*}
\lim_{T \to \infty} \P^\mu (B_{T, \eps}^c) = 0.
\end{align*} By \eqref{WM2}, this implies \eqref{81}, and hence finishes the proof.

\end{proof}

\beg{proof}[Proof of Theorem \ref{TN4}] %By Theorem \ref{TN3}, it suffices to prove for $d\in (2,4).$
We only prove the first assertion, as the second can be proved in the same way.

(a)  For any $r\in (0, \infty)$, we define
\begin{align*}
&\bar \mu_{T, r} : = \frac1{T - r} \int_{r}^{T} \delta_{X_t} \d t, \ \ T > r,\\
 &\tt \mu_{T, r} : = \frac1T \int_{r}^{T + r} \delta_{X_t} \d t, \ \ T > 0.
\end{align*}
 By \eqref{HS3}, there exists a constant  $c(r)\in (0,\infty)$  such that
 \beq\label{CR} \nu_r:= \nu  P_r^*\le c(r) \mu,\ \ \nu\in \scr P.\end{equation}
Then the Markov property, \eqref{UP1}  and  \eqref{UP3'} imply
 \beg{align*}&\lim_{T\to\infty} \sup_{\nu\in \scr P} \E^\nu\big[|\{T\W_2^2( \bar \mu_{T,r},\mu)-\bar\Xi_{r}(T)\}^+|^q\big]\\
 &= \lim_{T\to\infty} \sup_{\nu\in \scr P} \E^{\nu_r}\big[|\{T\W_2^2(  \mu_{T-r},\mu)- \Xi  (T-r)\}^+|^q\big]\\
&\le c(r) 2^{q-1}  \lim_{T\to\infty}   \E^\mu  \big[|\{(T-r)\W_2^2( \mu_{T-r},\mu)- \Xi (T-r)\}^+|^q+ r^q\W_2^{2q}(\mu_{T-r},\mu)\big] =0.\end{align*}
Similarly,
\beg{align*}&\lim_{T\to\infty} \sup_{\nu\in \scr P} \E^\nu\big[|\{T\W_2^2( \tt \mu_{T,r},\mu)-\tt\Xi_{r}(T)\}^+|^q\big]\\
 &= \lim_{T\to\infty} \sup_{\nu\in \scr P} \E^{\nu_r}\big[|\{T\W_2^2(  \mu_{T},\mu)- \Xi (T)\}^+|^q\big]\\
&\le c(r)    \lim_{T\to\infty}   \E^\mu  \big[|\{T \W_2^2( \mu_{T},\mu)- \Xi (T)\}^+|^q  =0.\end{align*}
So, it remains to verify that for any   $r \in (0, \infty)$ and $q\in [1,\infty)$,
\beq\label{ZD}\lim_{T\to\infty} T  \sup_{\nu\in \scr P} \Big(\E^\nu\big[|\W_2^2(\mu_T,\mu)-\W_2^2(\bar\mu_{T,r},\mu)|^q\big]\Big)^{\ff 1 q}=0, \end{equation}
\beq\label{ZD'}\lim_{T\to\infty} T  \sup_{\nu\in \scr P} \Big(\E^\nu\big[|\W_2^2(\tt \mu_{T,r},\mu)-\W_2^2(\bar\mu_{T,r},\mu)|^q\big]\Big)^{\ff 1 q}=0.\end{equation}

(b) Verification of \eqref{ZD}.  Noting that  for $\mu_r:=\ff 1 r \int_0^r \dd_{X_t}\d t$, we have
\beq\label{*M}\mu_T= (\theta-rT^{-1}) \bar\mu_{T,r}+rT^{-1}\mu_r+ (1-\theta)\bar\mu_{T,r},\ \ \theta\in (rT^{-1},1),\end{equation}
 by \eqref{convexity} and the triangle inequality, we obtain
 \beq\label{*M2} \beg{split}&\W_2(\mu_T,\bar\mu_{T,r}) \le \ss \theta \W_2\big((1 - \theta^{-1} r T^{-1})\bar\mu_{T,r} + \theta^{-1} r T^{-1}\mu_r,\ \bar\mu_{T,r}\big)\\
&\le \ss\theta \Big[\W_2\big((1-\theta^{-1} r T^{-1})\bar\mu_{T,r} + \theta^{-1} r T^{-1}\mu_r,\ (1-\theta^{-1} r T^{-1}) \mu  + \theta^{-1} r T^{-1}\mu_r\big)\\
&\qquad \quad +
\W_2\big((1- \theta^{-1} r T^{-1}) \mu  + \theta^{-1} r T^{-1}\mu_r,\ \mu\big)+ \W_2\big(\bar\mu_{T,r},\mu \big)\Big],\ \ \theta\in (rT^{-1},1).\end{split}\end{equation}
By \eqref{convexity}, we obtain
\beq\label{*M3} \W_2\big((1-\theta^{-1} r T^{-1})\bar\mu_{T,r} + \theta^{-1} r T^{-1}\mu_r,\ (1-\theta^{-1} r T^{-1}) \mu  + \theta^{-1} r T^{-1}\mu_r\big)\le \W_2(\bar\mu_{T,r},\mu),\end{equation}
and  by Lemma \ref{LC}, we find a constant $c_1\in (0,\infty)$ such that for any $T > r$ and $\theta \in (r T^{-1}, 1)$,
\beq\label{*M4}
\begin{split}
\W_2\big((1-\theta^{-1} r T^{-1}) \mu  + \theta^{-1} r T^{-1}\mu_r,\ \mu\big)\le c_1 \cdot \begin{cases}
\theta^{-1} r T^{-1}, &\textrm{if } d \in [1, 2),\\
\theta^{-1} r T^{-1} \log(1 + \theta r^{-1} T), &\textrm{if }  d = 2,\\
(\theta^{-1} r T^{-1})^{\frac12 + \frac1{d}}, &\textrm{if } d \in (2, 4).
\end{cases}
\end{split}
\end{equation}
Moreover, by \eqref{CR}, the Markov property and \eqref{UP1}, we find a constant  $c_2(r) >0$  such that
\beq\label{*GB}
\sup_{\nu\in \scr P} \big(\E^\nu[\W_2^{2q}(\bar\mu_{T,r},\mu)]\big)^{\ff 1 q} = \sup_{\nu \in \scr P} \big(\E^{\nu_r}[\W_2^{2q}(\mu_{T -r}, \mu)]\big)^{\ff 1 q}\le  c_2(r) T^{-1},\ \ T \geq 2r+2.
\end{equation}
Combining \eqref{*M2}-\eqref{*GB} together, we find a constant  $c_3 > 0$, depending on $r \in (0, \infty)$, such that for any $T \geq 2r + 2$ and $\theta \in (r T^{-1}, 1)$,
\begin{align*}
\sup_{\nu\in \scr P} \big(\E^\nu[\W_2^{2q}(\mu_T,\bar\mu_{T,r})]\big)^{\ff 1 q} \le c_3 \cdot \begin{cases}
\theta^{-1} T^{-2} + \theta T^{-1}, &\textrm{if } d \in [1, 2),\\
\theta^{-1} T^{-2} \log^2(\theta T) + \theta T^{-1}, &\textrm{if }  d = 2,\\
\theta^{-\frac{2}{d}} T^{-\frac{d + 2}{d}} + \theta T^{-1}, &\textrm{if } d \in (2, 4).
\end{cases}
\end{align*}
Choosing
\begin{align*}
\theta = \begin{cases} T^{-\frac12}, & \textrm{if } d \in [1, 2],\\
 T^{-\frac{2}{d + 2}}, &\textrm{if } d \in (2, 4),\end{cases}
\end{align*}
which is in $(rT^{-1},1)$ for large $T>0$, we get
\begin{align*}
\sup_{\nu\in \scr P} \big(\E^\nu[\W_2^{2q}(\mu_T,\bar\mu_{T,r})]\big)^{\ff 1 q} \les \begin{cases}
T^{-\frac32}, &\textrm{if } d \in [1, 2),\\
T^{-\frac32} \log^2(T), &\textrm{if }  d = 2,\\
T^{-\frac{d + 4}{d+2} }, &\textrm{if } d \in (2, 4),
\end{cases}
\ \ \ q \in [1, \infty).
\end{align*}
Combining this with \eqref{*GB}, and applying the triangle inequality and H\"older's inequality, we arrive at
\beg{align*} & T\sup_{\nu\in \scr P} \big(\E^\nu \big[ |\W_2^{2 }(\mu_T,\mu)- \W_2^2(\bar\mu_{T,r},\mu)|^q \big]\big)^{\ff 1 q}\\
&\le T\sup_{\nu\in \scr P} \big(\E^\nu \big[ |\W_2^{2 }(\mu_T,\bar\mu_{T,r})+ 2  \W_2 (\mu_T,\bar\mu_{T,r}) \W_2 (\bar\mu_{T,r},\mu)  |^q \big]\big)^{\ff 1 q}\\
& \le T \sup_{\nu\in \scr P} \big(\E^\nu[\W_2^{2q}(\mu_T,\bar\mu_{T,r})]\big)^{\ff 1 q} + 2 T \sup_{\nu\in \scr P} \big(\E^\nu[\W_2^{2q}(\mu_T,\bar\mu_{T,r})]\big)^{\ff 1 {2q}} \cdot \big(\E^\nu[\W_2^{2q}(\bar\mu_{T,r},\mu)]\big)^{\ff 1 {2q}}\\
& \les   T^{-\frac{1}{d \vee 2 + 2}} \big( \log (T) 1_{\{d = 2  \}} + 1_{\{d \ne 2  \}} \big)
\to 0\ \text{ as } \ T\to\infty.\end{align*}
So, \eqref{ZD} holds.

(c) Verification of \eqref{ZD'}. Let $\hat\mu_r:=\ff 1 r\int_T^{T+r} \dd_{X_t}\d t.$ We have
$$\tt\mu_{T,r}= (\theta-rT^{-1}) \bar\mu_{T,r} + rT^{-1} \hat\mu_r +(1-\theta) \bar\mu_{T,r},\ \ \theta\in (rT^{-1},1).$$
  By using this formula   replacing \eqref{*M}, the above argument leads to \eqref{ZD'}. Therefore, the proof is finished.

\end{proof}

\end{document}